\documentclass[11pt]{article}

\usepackage[margin=1.15in]{geometry}
\usepackage{amsmath,amssymb,amsthm,mathtools}
\usepackage{enumitem}
\usepackage[round,authoryear]{natbib}
\usepackage[
    colorlinks=true,
    linkcolor=blue,
    citecolor=blue,
    urlcolor=blue,
    hypertexnames=false,
    pdftitle={Doubling the dimension yields a benign landscape for the squared-stress},
    pdfauthor={Christopher Criscitiello},
    pdfsubject={Optimization landscapes for Euclidean distance geometry},
    pdfkeywords={Euclidean distance geometry, squared stress, benign landscapes,
                 nonconvex optimization, low-dimensional relaxation, matrix sensing}
]{hyperref}
\usepackage{mathrsfs}

\newtheorem{theorem}{Theorem}[section]
\newtheorem{lemma}[theorem]{Lemma}

\theoremstyle{definition}

\newtheorem{remark}[theorem]{Remark}

\newcommand{\Proj}{\operatorname{Proj}}
\newcommand{\R}{\mathbb R}
\newcommand{\RR}{\mathscr R}
\newcommand{\PP}{\mathscr P}
\newcommand{\ones}{\mathbf 1}
\newcommand{\cU}{\mathcal U}
\newcommand{\Y}{\mathcal Y}

\newcommand{\ip}[2]{\left\langle #1,#2\right\rangle}
\newcommand{\norm}[1]{\left\|#1\right\|}
\newcommand{\fro}[1]{\left\|#1\right\|_{\mathrm F}}

\newcommand{\tr}{\operatorname{Tr}}
\newcommand{\rank}{\operatorname{rank}}
\newcommand{\diag}{\operatorname{diag}}
\newcommand{\Diag}{\operatorname{Diag}}
\newcommand{\im}{\operatorname{im}}
\newcommand{\Sym}{\operatorname{Sym}}
\newcommand{\PSD}{\operatorname{PSD}}
\newcommand{\argmax}{\operatorname*{arg\,max}}
\newcommand{\lammax}{\lambda_{\max}}
\newcommand{\lammin}{\lambda_{\min}}
\newcommand{\Pc}{P_c}

\title{Doubling the dimension yields a benign landscape for the squared-stress}
\author{
Christopher Criscitiello\thanks{
The Wharton School, University of Pennsylvania, USA.
\texttt{crisciti@wharton.upenn.edu}.
}
}
\date{August 17, 2026}

\begin{document}

\maketitle

\begin{abstract}
We consider the Euclidean distance geometry problem (EDG): given a subset of the
pairwise distances of an unknown cloud of \(n\) points in \(\R^\ell\), recover the point cloud
up to rigid motions. When \(n\) is large, a popular practical approach is to minimize a
nonconvex quartic, known as the \emph{squared-stress} or \emph{s-stress}, over point clouds in \(\R^k\), with
\(k\) potentially larger than \(\ell\). It is a long-standing open problem to understand the
optimization landscape of the s-stress when all pairwise distances are known
\citep{malone2000sstress,parhizkar2013euclidean}. It was recently shown that the
landscape is not benign when \(k=\ell\), and it was conjectured that the landscape
becomes benign as soon as \(k\ge \ell+1\)
\citep{song2024localupdated,criscitiello2025snl}.

Here, we show that the complete-graph s-stress has a benign landscape whenever \(k\ge 2(\ell+1)\), establishing the conjecture up to a factor of two.
A key
idea is to view second-order criticality as a containment of two ellipsoids; finding a
descent direction then corresponds to finding a separating hyperplane that violates this
containment. This dual perspective yields the stated landscape result, and also applies
to any measurement operator whose inverse satisfies a simple frame condition.
\end{abstract}


\medskip


\tableofcontents

\section{Introduction}

We consider the Euclidean distance geometry problem (EDG). For a ground truth
configuration of $n$ points $z_1^\star,\ldots,z_n^\star\in\R^\ell,$ where \(\ell\ge1\),
we observe noiseless Euclidean distances $d_{ij}=\|z_i^\star-z_j^\star\|$
for pairs $\{i,j\}$ belonging to an edge set $E$. The goal is to recover the ground truth
configuration from these distances, up to rigid transformations. EDG has applications in
molecular conformation, wireless sensor networks, statics, dimensionality reduction, and
robotics; see, for example, \citet{liberti2014euclidean}, \citet{liberti2020distance}, and
\citet{mucherino2012distance}.  In most of these applications, $n$ is large while $\ell$ is small, e.g., $\ell \in \{2,3\}$.

A widely used nonconvex formulation for EDG is the \emph{squared-stress}, or \emph{s-stress},
objective \citep{takane1977nonmetric,groenen1996least}. Given an optimization
dimension $k$, one solves
\begin{align}\label{eq:s-stress}
    \min_{z_1,\ldots,z_n\in\R^k}
    s(z_1,\ldots,z_n)
    \qquad
    \text{with}
    \qquad
    s(z_1,\ldots,z_n)
    =
    \frac12
    \sum_{\{i,j\}\in E}
    \left(\norm{z_i-z_j}^2-d_{ij}^2\right)^2 .
    \tag{s-stress}
\end{align}
The obvious choice is $k=\ell$, but one may also overparameterize by optimizing in a
larger dimension $k>\ell$.

In this paper, we focus on the complete graph, meaning that all pairwise distances are
known.\footnote{Extending these results to incomplete graphs is left for future work. A promising approach is outlined in~\citep[\S8.8]{Criscitiello2025thesis}. We also note that the conjectured benign landscape for \(k\ge \ell+1\) is false for incomplete graphs: it already fails when \(E\) is missing a single edge~\citep[\S8.7.3]{Criscitiello2025thesis}.}
In that setting, EDG can be solved by classical multidimensional scaling
\citep{shepard1962analysis,torgerson1958theory,kruskal1964multidimensional,
kruskal1964nonmetric}: one recovers the centered Gram matrix via a linear transformation and then computes an
eigendecomposition.
Our
interest, however, is different: we study the landscape of the nonconvex s-stress itself.

It has been a long-standing question whether the complete-graph s-stress has a benign
landscape \citep{malone2000sstress,parhizkar2013euclidean}. The answer is negative
when $k=\ell$: even for \(n=\ell+2\) points, the complete-graph s-stress may admit spurious local minimizers
\citep{song2024localupdated,criscitiello2025snl}. 
On the other hand, based on numerical experiments, \citet{criscitiello2025snl}
conjecture that these spurious local minimizers disappear as soon as one relaxes by a
single dimension, namely \(k\ge \ell+1\). We establish this conjecture up to a factor of two:
we prove that the complete-graph s-stress has no spurious local minima as soon as
\(k\ge 2(\ell+1)\), regardless of the ground truth.

\subsection*{Benign landscapes and low-dimensional relaxations}

We say that the s-stress has a \emph{benign landscape} if all second-order critical
configurations are global minimizers (see Section~\ref{sec:optimalityconditions}). In particular, there are no spurious local minima, and all saddle points are strict.

A main motivation for studying benign landscapes is algorithmic. Spurious local minimizers can trap local search methods and prevent them from reaching a global solution. In contrast, for problems with benign landscapes, many local search algorithms---including gradient descent and trust-region methods---are provably guaranteed to find global minimizers~\citep{shub1987book,helmke1996optimization,chijinsaddles2017,
jin2018agdescapes,jin2019escape,lee2019strictsaddles,cartis2012complexity,
boumal2016globalrates}.

It is well known that increasing the number of parameters in a nonconvex problem can
eliminate spurious local minimizers and can sometimes produce a benign landscape
\citep{sdplr,burer2005local,bandeira2016lowrankmaxcut,safran2021effects}. In EDG,
it has been observed empirically that allowing the optimization to proceed in dimension $k$
larger than the true embedding dimension $\ell$ substantially improves numerical performance
\citep{fangoleary2011,TasissaLai2019ExactReconstruction,tangtoh2023,
smith2024riemannian,criscitiello2025snl}.

This is the point of view we adopt: we relax the optimization dimension from \(\ell\)
to \(k>\ell\) and ask when the landscape becomes benign. Crucially, we seek the
smallest possible \(k\), since the relaxed problem optimizes over an \(n\times k\)
matrix. As a point of comparison, common SDP-based methods for EDG optimize over
dense \(n\times n\) matrices
\citep{biswas2004semidefinite,biswas2006semidefinite,so2007theory,
zhu2010universal,TasissaLai2019ExactReconstruction}, which becomes prohibitive when
\(n\) is large. Thus, the goal is to prove benign landscape results while keeping
\(k\ll n\).

\subsection*{Main result}

Our main result is the following.

\begin{theorem}[Main result]
\label{thm:intro-informal}
If all pairwise distances are known (i.e., $E$ is complete) and $k\ge 2(\ell+1)$, then~\eqref{eq:s-stress} has a benign landscape, regardless of the ground truth $z_1^\star,\ldots,z_n^\star\in\R^\ell$.
\end{theorem}

Because the complete graph is trivially universally rigid \citep{gortlerThurston2014UniversalRigidity},
every global minimizer of the complete-graph s-stress corresponds to the ground truth
up to rigid motion, even when \(k>\ell\).

Theorem~\ref{thm:intro-informal} is a direct consequence of the more general
Theorem~\ref{thm:main}. The latter is formulated for measurement operators whose
inverses satisfy a simple frame condition. This abstraction 
isolates the precise structural properties of the
s-stress objective used in the proof. 

The remainder of the paper is organized as follows.
Section~\ref{sec:complete-edg} reviews
complete-graph EDG. Section~\ref{sec:frame} introduces the structured-inverse
setting, states the main landscape theorem (Theorem~\ref{thm:main}), and
reduces its proof to a key intermediate result, Theorem~\ref{thm:intermediate}.
The proof of Theorem~\ref{thm:intermediate} is then developed in
Section~\ref{sec:proof-intermediate}.
A key ingredient is the dual perspective introduced in Section~\ref{sec:directions}, where descent directions are defined variationally and second-order criticality is interpreted as an ellipsoid-containment condition.

Finally, by refining the analysis of Section~\ref{sec:proof-intermediate},
Section~\ref{sec:meq1} proves that the landscape of the s-stress is benign at the conjectured
threshold \(k=\ell+1\), provided \(n\le \ell+3\). Although this regime is not
computationally attractive, since \(k\) is of the same order as \(n\), it provides
further theoretical evidence for the conjecture that a single dimension of
relaxation suffices. 
The proof also introduces \emph{Schur-companion directions},
which supplement the kernel directions used in the main argument and may provide
a useful ingredient for resolving the full \(k\ge\ell+1\) conjecture.

\section{Related work}\label{sec:relatedwork}

The literature on Euclidean distance geometry is vast, as is the literature on benign
landscapes. We focus only on the works most closely related to the present paper.
For additional references and historical background on EDG, we refer the reader to \citep[Sec.~2]{criscitiello2025snl}; for a broader overview of benign landscape results, see~\citep[Ch.~6]{Criscitiello2025thesis}.

\paragraph{Landscapes for the s-stress.}
The first (and, to date, only) low-dimensional landscape guarantees for the s-stress were established by~\citet{criscitiello2025snl}. 
For the complete graph, they showed that the landscape is benign for arbitrary ground truths whenever \(k\gtrsim \ell+\sqrt{n\ell}\), and, for isotropic random ground truths, whenever \(k\gtrsim \ell+\log n\) with high probability.  The present paper builds on the analysis underlying the former result. In particular, several ingredients of our first-order analysis
(Section~\ref{sec:first-order}) are adapted from~\citet{criscitiello2025snl};
see also Remark~\ref{rem:comparisontoolddescentdirections}, where we explain that
the descent directions of~\citet{criscitiello2025snl} arise as a natural proxy
for the descent directions introduced in the present paper.

The guarantee for isotropic random ground truths 
was subsequently extended in~\citep[\S8.8]{Criscitiello2025thesis} to certain incomplete graphs. Specifically, if $E$ in~\eqref{eq:s-stress} is obtained from the complete graph by removing $r$ edges, then the landscape remains benign with high probability provided $k\gtrsim (1+r^3)(\ell+\log n).$

\paragraph{Related benign landscapes.}
EDG is an instance of low-rank matrix sensing~\citep{Chi_2019}, with positive semidefinite measurement
matrices \((e_i-e_j)(e_i-e_j)^\top\) for \(\{i,j\}\in E\).  Consequently, it is closely
related to a large literature establishing benign landscapes for quartic formulations of
matrix sensing, including matrix completion and phase retrieval.

For matrix completion, \citet{ge2016matrix,ge2017nospurious} prove that there are no spurious local
minima without relaxation (\(k=\ell\)), under incoherence assumptions on the ground truth. 
In the identity-sensing case \(L=I\) (see Remark~\ref{rem:idsensing}), their analysis
recovers the classical benign-landscape result for arbitrary ground truths without
relaxation. A recurring descent direction in this line of work is the Procrustes residual \(\dot Z = Z-Z_\star Q\), where \(Q\) is orthogonal, and is chosen by an orthogonal Procrustes alignment.

Similar residual directions also appear in phase retrieval.  For example,~\citet{sun2016geometric} use them to establish a benign landscape for Gaussian phase
retrieval.  More recently,~\citet{mcrae2025phaseretrievalmatrixsensingupdated} develops a
landscape framework for low-rank matrix sensing that exploits the positive semidefiniteness
of the measurement matrices. Rather than the classical Procrustes residual, his
approach uses randomized least-squares residual directions of the form $\dot Z=(Z_\star-ZR_{\rm ls})G,$ where \(R_{\rm ls}\) is the least-squares alignment (see
Remark~\ref{rem:lsresidual}) and \(G\) is a matrix of i.i.d.\ Gaussian entries.  This framework yields phase-retrieval guarantees with optimal statistical sample complexity.

\paragraph{Convex approaches.}
For incomplete graphs, the Euclidean distance geometry problem is NP-hard even under natural rigidity assumptions~\citep{eren2004rigidity,aspnes2004computational,aspnes2006theory}.
Nevertheless, under suitable assumptions, EDG admits tractable convex formulations based on semidefinite programming. In particular, SDP methods can recover \emph{universally rigid} configurations in polynomial time~\citep{biswas2004semidefinite,biswas2006semidefinite,so2007theory,zhu2010universal}.
Building on ideas from matrix completion, \citet{TasissaLai2019ExactReconstruction} also establish efficient recovery guarantees for a natural SDP from randomly sampled distance measurements, under suitable incoherence assumptions.

Despite their strong theoretical guarantees, these approaches optimize over dense \(n\times n\) matrix variables, leading to memory and computational costs that scale at least quadratically with \(n\).  This motivates low-dimensional nonconvex formulations that factor the Gram matrix and optimize directly over an \(n\times k\) variable with \(k\ll n\).

\section{Complete-graph Euclidean distance geometry}
\label{sec:complete-edg}

This section covers the basic objects for complete-graph EDG and their properties. For
more background on these objects, see \citet[Sec.~3]{criscitiello2025snl}.

\subsection{Notation}

Let $e_i\in\R^n$ denote the $i$th standard basis vector, and let $\ones\in\R^n$ denote
the all-ones vector. 
For a vector $x\in\R^n$, its $i$th entry is denoted interchangeably by $x_i$ or $x(i)$, whichever is more convenient.
Sometimes we denote $x$ by $(x_i)_i$.
The $(i,j)$-entry of a matrix $X$ is $X_{ij}$. 
The $k\times k$ identity matrix is denoted by $I_k$; when the dimension is clear from
context, we simply write $I$. We also use $I$ to denote the identity operator on a
space of matrices.

The set of $n\times n$ symmetric matrices is denoted $\Sym(n)$. Also denote
$$\mathrm{PSD}(n) = \{X \in \Sym(n) : X \succeq 0\}, \qquad \mathrm{PSD}_{\leq k}(n) = \{X \in \mathrm{PSD}(n) : \rank(X) \leq k\}.$$
For $X\in\PSD(n)$, we write $X^{1/2}$ for its unique positive semidefinite square root,
$X^\dagger$ for its Moore--Penrose pseudoinverse, and
$X^{\dagger/2}:=(X^\dagger)^{1/2}$.

For a linear subspace
$\cU\subseteq\R^n$, let $P_{\cU} \in \mathbb{R}^{n \times n}$ denote the orthogonal projector onto $\cU$, and define
\[
    \Sym(\cU):=\{Y\in\Sym(n):Y=P_{\cU}YP_{\cU}\}.
\]
If $X\in\Sym(n)$, we define the trace of the
compression of $X$ to $\cU$ by $\tr_{\cU}(X):=\tr(X P_{\cU}).$

For EDG, the relevant subspace is the centered subspace
\[
    1^\perp := \{x\in\R^n:\ones^\top x=0\}.
\]
The orthogonal projector onto $1^\perp$ is $\Pc=I-\frac1n\ones\ones^\top.$
The set of \emph{centered matrices} is
$$\mathrm{Cent}(n) := \Sym(1^\perp) = \{Y\in\Sym(n):Y\ones = 0\}.$$

For a matrix \(X\), let \(\diag(X)\) denote the vector of its diagonal entries. For a vector
\(x\), let \(\Diag(x)\) denote the diagonal matrix with diagonal \(x\); when convenient, we
instead write \(\Diag(x_i)\). Finally, we define $\Diag(X):=\Diag(\diag(X)).$

We always use Frobenius inner products, unless indicated otherwise:
\[
    \ip{X}{Y}=\tr(X^\top Y),
    \qquad
    \fro{X}^2=\tr(X^\top X).
\]
For $X\in\Sym(n)$, its smallest and largest eigenvalues are denoted
$\lammin(X)$ and $\lammax(X)$. The image of a matrix $X$, i.e., the span of its
columns, is denoted $\im(X)$; its kernel is denoted $\ker(X)$; and the dimension of its
kernel is denoted $\dim \ker(X)$.

\subsection{The EDM map}

Given a configuration of points $z_1,\ldots,z_n\in\mathbb{R}^k$, let
$Z\in\mathbb{R}^{n\times k}$ denote the matrix whose $i$th row is $z_i^\top$.
Its Gram matrix $Y=ZZ^\top$ records the pairwise inner products, with entries
$Y_{ij}=z_i^\top z_j$. The squared Euclidean distance between points $z_i$ and
$z_j$ is
\[
    \norm{z_i-z_j}^2=Y_{ii}+Y_{jj}-2Y_{ij}.
\]
Define the EDM map $\Delta\colon\Sym(n)\to\Sym(n)$, short for \emph{Euclidean distance matrix},
by
\begin{equation}\label{eq:EDM-map}\tag{EDM-map}
\begin{split}
    \Delta(Y)
    &=
    \frac12\big(\ones\,\diag(Y)^\top+\diag(Y)\,\ones^\top-2Y\big), \\
    \Delta(Y)_{ij}
    &=
    \frac12(Y_{ii}+Y_{jj}-2Y_{ij}).
\end{split}
\end{equation}
Thus $2\Delta(ZZ^\top)$ is the matrix of pairwise squared distances.

For complete-graph EDG,~\eqref{eq:s-stress} can now be written as
\begin{equation}\label{eq:s-stress-complete}\tag{s-stress-complete}
\begin{split}
    s(Z)&=\fro{\Delta(ZZ^\top-Z_\star^{} Z_\star^\top)}^2 \\
    &= \ip{ZZ^\top-Z_\star^{} Z_\star^\top}{\Delta^* \Delta(ZZ^\top-Z_\star^{} Z_\star^\top)},
    \end{split}
\end{equation}
where $\Delta^*$ is the adjoint of $\Delta$, with respect to the Frobenius inner product on $\Sym(n)$.

The map \(\Delta\) is not invertible on all of \(\Sym(n)\): its kernel consists of
matrices of the form
\[
    \ones x^\top+x\ones^\top,
    \qquad x\in\R^n.
\]
The orthogonal complement of this kernel is the space of centered matrices
\(\mathrm{Cent}(n)=\Sym(1^\perp)\). 
Restricting to centered matrices, $\Delta^\ast\Delta \colon \mathrm{Cent}(n) \to \mathrm{Cent}(n)$ is invertible and satisfies
\begin{align}\label{eq:EDM-inverse}\tag{EDM-inverse}
    (\Delta^\ast\Delta)^{-1}(X)=X-\Pc\Diag(X)\Pc.
\end{align}
This is the only specific formula involving \(\Delta\) that we need in order to prove Theorem~\ref{thm:intro-informal}.

The operator $\Delta^\ast\Delta$ has eigenvalues $1$, $n/2$, and $n$ on centered
matrices. Therefore, it is far from an
approximation of the identity when $n$ is large. Moreover, it fails the restricted
isometry property (RIP) commonly used to prove benign landscapes in low-rank matrix sensing~\citep{ge2016matrix,bhojanapalli2016global,zhang2018howmuchrip}. For the derivation of \eqref{eq:EDM-inverse}, the eigenanalysis of \(\Delta^* \Delta\), and the RIP failure calculation, see \citet[Sec.~3.3]{criscitiello2025snl}.

\section{A structured-inverse landscape result}
\label{sec:frame}

We now move from the specific EDM map to a more general class of measurement
operators. The common feature is that the inverse operator has the same
``identity minus frame operator'' form as~\eqref{eq:EDM-inverse}. At the end of
this section, we state a general theorem for any measurement operator with this property.

\subsection{The abstract problem and the structured-inverse hypothesis}\label{sec:symmetry}
Let $\cU\subseteq\R^n$ be a $p$-dimensional subspace, and let
\[
    L\colon\Sym(\cU)\to\Sym(\cU)
\]
be self-adjoint and positive definite. We extend $L$ to all of $\Sym(n)$ by
\[
    L(X):=L(P_{\cU}XP_{\cU}).
\]
Thus $L(X)\in\Sym(\cU)$, and the quadratic form induced by $L$ only sees the
$\cU$-block of $X$.

For arbitrary $Z\in\R^{n\times k}$ and $Z_\star\in\R^{n\times \ell}$, set
\[
    Y=ZZ^\top,
    \qquad
    Y_\star=Z_\star^{} Z_\star^\top,
\]
and consider
\begin{align}\label{eq:P}
    \min_{Z\in\R^{n\times k}}
    g(Z;Z_\star)
    :=
    \ip{ZZ^\top-Y_\star}{L(ZZ^\top-Y_\star)}.
    \tag{P}
\end{align}
Our goal is to understand the landscape of~\eqref{eq:P}. In the special case
$L=\Delta^*\Delta$ and $\mathcal U = 1^\perp$,
the objective $g$ coincides with the complete-graph s-stress: $g(Z;Z_\star)=s(Z)$; see~\eqref{eq:s-stress-complete}.

\paragraph{Symmetry.}
Since \(L\) first projects onto \(\Sym(\cU)\),
\begin{align}\label{eq:symmetry}
    g(Z;Z_\star)=g(P_{\cU}Z;P_{\cU}Z_\star).
\end{align}
Consequently, if \(Z\) is a second-order critical point of
\(g(\cdot;Z_\star)\), then \(P_{\cU}Z\) is a second-order critical point of
\(g(\cdot;P_{\cU}Z_\star)\), with the same objective value.
This also follows directly by inspecting the first- and second-order
criticality conditions \eqref{eq:1-criticality}--\eqref{eq:2-criticality} below.

Hence, without loss of generality, we may replace \(Z_\star\) by
\(P_{\cU}Z_\star\) and assume
\[
    Y_\star\in\Sym(\cU),
    \qquad
    Y_\star\succeq0,
    \qquad
    \rank(Y_\star)\le \ell.
\]
Moreover, to prove that~\eqref{eq:P} has no non-global second-order critical
points, it suffices to show that every second-order critical point \(Z\) whose
columns lie in \(\cU\) (equivalently, \(Y\in\Sym(\cU)\)) is globally optimal.

\subsubsection*{The structured-inverse hypothesis}

Assume there are vectors
\begin{align}\tag{frame-atoms}
    a_1,\ldots,a_N\in\cU
\end{align}
and a constant \(\eta\in(0,1]\) such that
\begin{align}
    \sum_{i=1}^N a_i^{} & a_i^\top \preceq P_{\cU},
    \tag{F1}\label{eq:F1}\\
    \norm{a_i}^2 &\le 1-\eta
    \qquad
    \forall i,
    \tag{F2}\label{eq:F2}\\
    L^{-1}=I-\Gamma,
    \qquad &
    \Gamma(X)=\sum_{i=1}^N a_i^{} a_i^\top(a_i^\top Xa_i).
    \tag{F3}\label{eq:F3}
\end{align}
Condition~\eqref{eq:F1} is an upper-frame (or Bessel) condition on the atoms \(\{a_i\}_{i=1}^N\). 
When equality holds in~\eqref{eq:F1}, it reduces to a Parseval tight-frame condition. We do not require the number of atoms \(N\) to coincide with the number of points \(n\), and we allow the degenerate case \(N=0\).

For complete-graph EDG,~\eqref{eq:EDM-inverse} shows that \(L=\Delta^*\Delta\) satisfies~\eqref{eq:F3} with
\begin{align}\label{eq:EDGinverseframeconditions}\tag{EDG-frame}
    \cU=1^\perp,
    \qquad
    a_i=\Pc e_i,
    \qquad
    i=1,\ldots,n.
\end{align}
In that case,
\[
    \sum_{i=1}^n a_i^{} a_i^\top=\Pc,
    \qquad
    \norm{a_i}^2=1-\frac1n,
\]
so~\eqref{eq:F1} holds with equality and~\eqref{eq:F2} holds with \(\eta=1/n\).

\begin{remark}
Conditions~\eqref{eq:F1}--\eqref{eq:F3} are closely related to the framework proposed
in~\citep[App.~A]{criscitiello2025snl}.  In particular, condition Q5 therein is precisely
the representation~\eqref{eq:F3}. The assumptions of~\citep[App.~A]{criscitiello2025snl}
are stronger, however, as they additionally require a quantitative conditioning property
(Q4 therein).
\end{remark}

\begin{lemma}[Basic properties of $\Gamma$]
\label{lem:gamma-properties}
For $X\in\Sym(\cU)$, the map $\Gamma$ given by~\eqref{eq:F1}--\eqref{eq:F3} satisfies:
\begin{enumerate}[label=(\roman*)]
    \item if $X\succeq0$, then $\Gamma(X)\succeq0$;
    \item if $X\succeq0$, then $\tr\Gamma(X)\le(1-\eta)\tr X;$
    \item if \(\lambda_{\max}(X)\geq0\), then $\lambda_{\max}(\Gamma(X))
        \le
        (1-\eta)\lambda_{\max}(X);$
    \item $I-\Gamma \colon \Sym(\cU) \to \Sym(\cU)$ is positive definite.
\end{enumerate}
\end{lemma}
\begin{proof}
If $X\succeq0$, then $a_i^\top Xa_i\ge0$, and hence
\[
    \Gamma(X)=\sum_i a_i a_i^\top(a_i^\top Xa_i)\succeq0.
\]
This establishes \emph{(i)}, i.e., $\Gamma$ is completely positive.  For the trace contraction \emph{(ii)},
\[
    \tr\Gamma(X)
    \stackrel{\eqref{eq:F3}}{=}
    \sum_i\norm{a_i}^2(a_i^\top Xa_i)
    \stackrel{\eqref{eq:F2}}{\le}
    (1-\eta)\sum_i a_i^\top Xa_i
    =
    (1-\eta)\bigg\langle \sum_i a_i a_i^\top, X \bigg \rangle
    \stackrel{\eqref{eq:F1}}{\leq}
    (1-\eta)\tr X.
\]
For the spectral contraction \emph{(iii)}, let \(X\in\Sym(\cU)\) satisfy
\(\lambda_{\max}(X)\geq 0\), and let \(y\in\cU\) be a unit vector. Therefore, $a_i^\top Xa_i\le \lambda_{\max}(X)\|a_i\|^2$, and
\[
\begin{aligned}
    y^\top\Gamma(X)y
    &=
    \sum_i (y^\top a_i)^2(a_i^\top Xa_i)
    \le
    \lambda_{\max}(X)\sum_i (y^\top a_i)^2\|a_i\|^2\\
    &\stackrel{\eqref{eq:F2}}{\le}
    (1-\eta)\lambda_{\max}(X)\sum_i (y^\top a_i)^2
    \stackrel{\eqref{eq:F1}}{\leq}
    (1-\eta)\lambda_{\max}(X).
\end{aligned}
\]
Taking the supremum over
unit \(y\in\cU\) gives $\lambda_{\max}(\Gamma(X))
    \le
    (1-\eta)\lambda_{\max}(X).$

Finally, let us show \emph{(iv)}.  For each $i$,
\[
    (a_i^\top Xa_i)^2
    =
    \ip{Xa_i}{a_i}^2
    \le
    \norm{Xa_i}^2\norm{a_i}^2
    =
    (a_i^\top X^2a_i)\norm{a_i}^2
    \stackrel{\eqref{eq:F2}}{\leq}
    (1-\eta)(a_i^\top X^2a_i),
\]
by Cauchy--Schwarz. Hence
\[
    \ip{X}{\Gamma(X)}
    \stackrel{\eqref{eq:F3}}{=} \sum_i(a_i^\top Xa_i)^2
    \le
    (1-\eta)\sum_i a_i^\top X^2a_i
    =
    (1-\eta)\bigg\langle\sum_i a_i a_i^\top, X^2 \bigg \rangle
    \stackrel{\eqref{eq:F1}}{\leq}
    (1-\eta)\fro{X}^2.
\]
Thus
\[
    \ip{X}{(I-\Gamma)X}\ge\eta\fro{X}^2,
\]
so $I-\Gamma\succ0$.
\end{proof}

\subsection{Optimality conditions}\label{sec:optimalityconditions}
Define the following ``stress'' matrix:\footnote{
We use the term \emph{stress} because, in the special case \(L=\Delta^*\Delta\), the matrix \(S\) can be interpreted as a stress matrix in the sense of rigidity theory~\citep{connelly2005generic,gortlerHealyThurston2010GlobalRigidity,gortlerThurston2014UniversalRigidity}. In that setting, \(S\) is the graph Laplacian whose edge \(\{i,j\}\) has weight $\frac12\|z_i^\star-z_j^\star\|^2-\frac12\|z_i-z_j\|^2.$
}
\begin{align}\label{eq:S}\tag{def-$S$}
    S:=L(Y_\star-Y).
\end{align}
The first- and second-order criticality conditions for~\eqref{eq:P} can be expressed in terms of \(S\) as follows.

We say that $Z$ is a \emph{first-order critical point} for~\eqref{eq:P} if
\begin{align}\label{eq:1-criticality}
    SZ=0.
    \tag{1-criticality}
\end{align}
Likewise, $Z$ is a \emph{second-order critical point} if it is first-order critical and additionally satisfies
\begin{align}\label{eq:2-criticality}
    0\le
    -\ip{\dot Z\dot Z^\top}{S}
    +
    \frac12\ip{\dot Y}{L\dot Y},
    \qquad
    \dot Y=Z\dot Z^\top+\dot Z Z^\top,
    \tag{2-criticality}
\end{align}
for all perturbations $\dot Z$.
These conditions are obtained by setting the gradient of the objective in~\eqref{eq:P} to zero and requiring the Hessian to be positive semidefinite.

Let us record two general facts about problem~\eqref{eq:P}. Both hold for an arbitrary positive semidefinite operator $L \colon \Sym(\cU)\to\Sym(\cU)$. Together, they show that second-order critical points $Z$ are already globally optimal in two important regimes: when the dimension $k$ is sufficiently large, and when $Z$ is rank deficient.

\paragraph{The convex regime.} 
The first regime is when the factorization dimension is sufficiently large. In this case, the rank constraint becomes vacuous and the problem reduces to a convex optimization problem.  Indeed,~\eqref{eq:P} admits the equivalent formulation
\begin{align}\label{eq:downstairs}
  \min_Y \,
  f(Y;Y_\star)
  :=
  \ip{Y-Y_\star}{L(Y-Y_\star)}
  \qquad
  \text{subject to}
  \qquad
  Y\in\PSD_{\le k}(n)\cap\Sym(\cU).
\end{align}
Since $L\succeq0$, the objective $f(\cdot;Y_\star)$ is convex.

Moreover, every element of $\Sym(\cU)$ has rank at most $p=\dim(\cU).$
Consequently, if $k\ge p$, then
\[
\PSD_{\le k}(n)\cap\Sym(\cU)
=
\PSD(n)\cap\Sym(\cU),
\]
and the rank constraint disappears. The feasible set is therefore convex, and hence~\eqref{eq:downstairs} is a convex optimization problem.

Using standard arguments (see, e.g.,~\citealp[Prop.~2.7,~3.32]{levin2022lifts}), every second-order critical point of the factorized problem~\eqref{eq:P} maps to a stationary point of~\eqref{eq:downstairs}. Since every stationary point of a convex optimization problem is globally optimal, it follows that every second-order critical point of~\eqref{eq:P} is a global minimizer whenever $k\ge p$.  
In summary:

\begin{lemma}\label{lem:kgeqp}
Assume $k\ge p$. If $Z \in \mathbb{R}^{n \times k}$ is second-order critical for~\eqref{eq:P}, then $Z$ is a global minimizer of~\eqref{eq:P}.
\end{lemma}
\noindent For a different proof of Lemma~\ref{lem:kgeqp} in the EDG setting, see~\citep{song2024localupdated}.

\paragraph{Rank-deficient configurations.} 
The second regime is when $Z$ fails to have full
column rank. 

\begin{lemma}\label{lem:rkdeficient}
If $Z \in \mathbb{R}^{n \times k}$ is second-order critical for~\eqref{eq:P} and $\rank(Z)<k$, then $Z$ is a global minimizer of~\eqref{eq:P}.
\end{lemma}
\noindent
This is a standard fact in low-rank optimization. For a proof, see~\citet[Thm.~7]{journee2010low} or~\citet[Lem.~3.4]{criscitiello2025snl}.

\subsection{The structured-inverse theorem}

Lemmas~\ref{lem:kgeqp} and~\ref{lem:rkdeficient} reduce the analysis to full-rank
second-order critical points with \(k<p\). 
Our central result is the following
theorem, whose proof is carried out in
Section~\ref{sec:proof-intermediate}.

\begin{theorem}[Full-rank 2-critical points]
\label{thm:intermediate}
Assume \(L\) satisfies~\eqref{eq:F1}--\eqref{eq:F3}, assume \(k<p\), and suppose
\(Y,Y_\star\in\Sym(\cU)\). If \(Z\) is a full-rank second-order critical point
of~\eqref{eq:P}, then either \(Z\) is globally optimal, or
\[
    k-\ell
    \le
    (1-\eta)\bigl(\rank(W^\top Y_\star W)+2\bigr),
\]
where the columns of \(W\in\R^{n\times(p-k)}\) form an orthonormal basis for
\(\im(Z)^\perp\cap\cU\).
\end{theorem}

Once Theorem~\ref{thm:intermediate} is established, the following structured-inverse landscape result follows.

\begin{theorem}
\label{thm:main}
Assume \(L\) satisfies~\eqref{eq:F1}--\eqref{eq:F3}. If either \(k\ge p\), or \(k<p\)
and
\begin{align}\label{eq:assumedinequ}
    k-\ell > (1-\eta)(\min\{\ell,p-k\}+2),
\end{align}
then every second-order critical point of~\eqref{eq:P} is globally optimal.
In particular, the landscape of~\eqref{eq:P} is benign whenever $k\ge 2(\ell+1).$
\end{theorem}

\begin{proof}
If \(k\ge p\), the claim follows from Lemma~\ref{lem:kgeqp}. Thus assume \(k<p\),
and let \(Z\) be a second-order critical point. If \(\rank(Z)<k\), then
Lemma~\ref{lem:rkdeficient} implies that \(Z\) is globally optimal. Hence we may
assume that \(Z\) has full rank. By the symmetry reduction~\eqref{eq:symmetry}, we
may also assume that $Y = ZZ^\top \in \Sym(\cU)$ and $Y_\star = Z_\star^{} Z_\star^\top \in \Sym(\cU)$.

Suppose for contradiction that \(Z\) is not globally optimal. 
Then
Theorem~\ref{thm:intermediate} yields
\[
    k-\ell \le (1-\eta)(\rank(W^\top Y_\star W)+2).
\]
On the other hand,
\[
    \rank(W^\top Y_\star W)
    \le
    \min\{\rank(Y_\star),\,p-k\}
    \le
    \min\{\ell,\,p-k\}.
\]
Therefore
\[
    k-\ell \le (1-\eta)(\min\{\ell,\,p-k\}+2),
\]
contradicting the assumed inequality~\eqref{eq:assumedinequ}. So every second-order critical point
must be optimal.

Finally, if \(k\ge 2(\ell+1)\), then either \(k\ge p\), or else \(k<p\) and
\[
    k-\ell \ge \ell+2 \ge \min\{\ell,p-k\}+2
    > (1-\eta)(\min\{\ell,p-k\}+2),
\]
where the strict inequality uses \(\eta>0\). Thus the theorem applies.
\end{proof}

Theorem~\ref{thm:intro-informal} is now immediate. Indeed, for complete-graph EDG, \(L = \Delta^* \Delta\) satisfies
the structured-inverse hypotheses~\eqref{eq:F1}--\eqref{eq:F3}, see~\eqref{eq:EDGinverseframeconditions}.

\begin{remark}[Least-squares residual]\label{rem:lsresidual}
The quantity \(W^\top Y_\star W\) appearing in Theorem~\ref{thm:intermediate} has a simple geometric meaning. Define the ``best least-squares linear transformation''

\[
    R_{\rm ls}
    :=
    {\arg\min}_{R\in\R^{k\times \ell}}\|Z_\star-Z R\|_{\mathrm F}^2
    =
    (Z^\top Z)^{-1} Z^\top Z_\star,
\]
and define the \emph{least-squares residual}
\[
    Z_{\rm ls}:=Z_\star-ZR_{\rm ls} = (I - Z (Z^\top Z)^{-1} Z^\top) Z_\star = W W^\top Z_\star.
\]
Therefore,
\[
    Z_{\rm ls}^{} Z_{\rm ls}^\top
    =
    WW^\top Y_\star WW^\top,
    \qquad
    W^\top Y_\star W
    =
    (W^\top Z_{\rm ls})(W^\top Z_{\rm ls})^\top .
\]
Thus \(\rank(W^\top Y_\star W)=\rank(Z_{\rm ls})\): loosely speaking, it measures the dimension of the component
of the ground truth that cannot be explained by the best linear alignment of \(Z\) to
\(Z_\star\). This residual viewpoint also appears in the geometric interpretation of the
descent directions in~\citep[App.~D]{criscitiello2025snl}.
\end{remark}

\begin{remark}[Identity sensing]\label{rem:idsensing}
Theorem~\ref{thm:main}
includes identity sensing \(L=I\) as a degenerate special case. Indeed, take
\(N=0\), so that \(\Gamma=0\), and hence $L^{-1}=I-\Gamma=I.$
Condition~\eqref{eq:F1} is then vacuous, while~\eqref{eq:F2}--\eqref{eq:F3}
hold with \(\eta=1\). So Theorem~\ref{thm:main} implies that the
landscape of~\eqref{eq:P} is benign whenever \(k>\ell\).
However, this guarantee is loose.  In the identity-sensing setting, it is known that
the landscape is already benign when \(k=\ell\), for arbitrary ground truths~\citep{ge2017nospurious}. Possible
approaches for closing this gap are suggested in Section~\ref{sec:meq1}.

More generally, it remains open whether condition~\eqref{eq:assumedinequ} in Theorem~\ref{thm:main} is tight over the full family of operators satisfying \eqref{eq:F1}--\eqref{eq:F3}. 
Section~\ref{sec:meq1} shows the condition is not tight when $m = p - k = 1$.
For the operator \(L=\Delta^*\Delta\), we do not expect the condition to be tight.
\end{remark}

\section{Proof of Theorem~\ref{thm:intermediate}}\label{sec:proof-intermediate}

This section is devoted to proving Theorem~\ref{thm:intermediate}.
Throughout, we therefore assume that \(L\) satisfies~\eqref{eq:F1}--\eqref{eq:F3},
that \(k<p\), that \(Y,Y_\star\in\Sym(\cU)\), and that \(Z\) has full column rank.

Sections~\ref{sec:first-order} to~\ref{sec:second-order} develop the main ingredients of the proof:
consequences of first-order criticality, the construction of descent
directions, and consequences of second-order criticality. The end of
Section~\ref{sec:second-order} combines these ingredients to prove
Theorem~\ref{thm:intermediate}.

\subsection{Consequences of first-order criticality}
\label{sec:first-order}

This section develops the consequences of first-order criticality.
Our arguments closely follow those of~\citep[\S4.1,~\S4.3]{criscitiello2025snl}. In particular, Lemmas~\ref{lem:stress-psd} and~\ref{lem:first-order-blocks} below extend~\citep[Lems.~4.2,~4.4]{criscitiello2025snl} to our structured-inverse setting~\eqref{eq:F1}--\eqref{eq:F3}.

Since \(Z\) has full rank, we may write
\begin{align}\label{eq:defVSigma}\tag{def-$V, \Sigma$}
    Z=V\Sigma^{1/2},
    \qquad
    V^\top V=I_k,
    \qquad
    \Sigma\succ0.
\end{align}
Let \(m:=p-k\), and choose \(W\in\R^{n\times m}\) so that \([V,W]\) is an orthonormal basis of \(\cU\), i.e.,
\begin{align}\label{eq:VVWW}\tag{def-$W$}
W^\top W = I_m \qquad \text{and} \qquad VV^\top + WW^\top = P_\cU.
\end{align}

Using~\eqref{eq:F3}, the definition of \(S\) in~\eqref{eq:S} can be rearranged as
\begin{align}\label{eq:exposingYstar}\tag{expose-$Y_\star$}
    Y_\star
    =
    Y+L^{-1}S
    =
    Y+S-\Gamma(S).
\end{align}
This identity is useful because it exposes the ground truth \(Y_\star\), allowing us to exploit the fact that $Y_\star \succeq 0$ and $\rank(Y_\star)\le \ell.$

\paragraph{Positive-semidefinite stress.} 
The first consequence of~\eqref{eq:exposingYstar} is that the stress matrix $S$ must be positive semidefinite.
This is useful because the first term in~\eqref{eq:2-criticality} becomes nonpositive, which opens the door to constructing descent directions.

\begin{lemma}
\label{lem:stress-psd}
At every first-order critical point $Z$, $S\succeq0.$
Consequently, at a full-rank first-order critical point,
\[
    S = W\Omega W^\top
    \qquad\text{for some}\qquad
    \Omega \in \mathrm{PSD}(m).
\]
\end{lemma}

\begin{proof}
For contradiction, suppose $S$ has a negative eigenvalue. Let
\[
    \gamma:=\lammax(-S)>0,
\]
and choose a unit eigenvector $x\in\cU$ such that $(-S)x=\gamma x$. Since $SZ=0$ by~\eqref{eq:1-criticality}, and since $x$ belongs
to a nonzero eigenspace of $S$, $x$ is orthogonal to $\im(Z)$. Hence, $Yx=ZZ^\top x=0$.
Appealing to~\eqref{eq:exposingYstar}, we obtain
\begin{align}\label{eq:tobecontradicted}
0\le x^\top Y_\star x
    =
    x^\top (Y+S-\Gamma(S)) x
    =
    -\gamma+x^\top\Gamma(-S)x
    \leq
    -\gamma+\lammax(\Gamma(-S)).
\end{align}
On the other hand, Lemma~\ref{lem:gamma-properties} \emph{(iii)} implies
\[
    \lammax(\Gamma(-S))
    \leq (1-\eta)\lammax(-S)
    <\gamma,
\]
which contradicts~\eqref{eq:tobecontradicted}. Therefore $S\succeq0$.

Since $S\succeq0$ and $SV=0$, the image of $S$ is contained in $\im(W)$. Hence
$S=W\Omega W^\top$ for some $\Omega\succeq0$.
\end{proof}

\paragraph{Block identities and a trace bound.} 
For each frame atom $a_i \in \cU$, define
\begin{align}\label{eq:defviwi}
    v_i:=V^\top a_i\in\R^k,
    \qquad
    w_i:=W^\top a_i\in\R^m,
    \qquad
    \tau_i:=\norm{w_i}^2.
    \tag{def-$v_i, w_i, \tau_i$}
\end{align}
When $L = \Delta^* \Delta$, $v_i$ and $w_i$ are the rows of $V$ and $W$, respectively.
By~\eqref{eq:F1},~\eqref{eq:F2} and~\eqref{eq:VVWW},
\begin{align}
    \sum_i v_i^{} v_i^\top &\preceq I_k,
    \qquad
    \sum_i w_i^{} w_i^\top\preceq I_m, \label{eq:sumvvtwwt} \\
    \norm{v_i}^2&+\norm{w_i}^2=\norm{a_i}^2\le1-\eta. \label{eq:viwinorm}
\end{align}

Now suppose \(Z\) is first-order critical. Then \(SV=0\) by~\eqref{eq:1-criticality}, and Lemma~\ref{lem:stress-psd} yields $S=W\Omega W^\top.$
Taking the $V,V$-block of~\eqref{eq:exposingYstar} gives
\begin{align}\label{eq:M}
    M:=V^\top Y_\star V
    = \Sigma-P
    \tag{$V, V$-block}
\end{align}
where we define
\begin{align}\label{eq:defP}\tag{def-$P$}
    P:=V^\top\Gamma(S)V=\sum_i d_i\,v_i^{} v_i^\top, \qquad d_i:=w_i^\top\Omega w_i.
\end{align}
Taking the $W,W$-block of~\eqref{eq:exposingYstar} gives
\begin{align}\label{eq:K}
    K:=W^\top Y_\star W
    = \Omega - W^\top \Gamma(S) W
    = \Omega-\sum_i d_i\,w_i^{} w_i^\top,
    \tag{$W,W$-block}
\end{align}
The matrices $M, P, K$---defined via~\eqref{eq:M},~\eqref{eq:defP},~\eqref{eq:K}---play a central role in the analysis that follows.

\begin{lemma}
\label{lem:first-order-blocks}
At every full-rank first-order critical point $Z$,
\[
    0\preceq K\preceq\Omega, \qquad \text{and} \qquad \tr P\le(1-\eta)\tr K.
\]
\end{lemma}

\begin{proof}
Since $Y_\star\succeq0$, $K = W^\top Y_\star W \succeq 0$. By Lemma~\ref{lem:gamma-properties} \emph{(i)} and Lemma~\ref{lem:stress-psd}, $W^\top \Gamma(S) W \succeq 0$.  Therefore, $K\preceq\Omega$ by~\eqref{eq:K}.

For the trace bound, using \(VV^\top+WW^\top=P_{\cU}\) and~\eqref{eq:K},
\[
    \tr P
    =
    \tr\Gamma(S)-\tr(W^\top\Gamma(S)W)
    =
    \tr\Gamma(S)-\tr\Omega+\tr K.
\]
By Lemma~\ref{lem:gamma-properties}~\emph{(ii)},
\[
    \tr\Gamma(S)
    \le
    (1-\eta)\tr S
    =
    (1-\eta)\tr\Omega.
\]
Therefore
\[
    \tr P
    \le
    \tr K-\eta\tr\Omega
    \le
    (1-\eta)\tr K,
\]
where the final inequality uses \(K\preceq\Omega\).
\end{proof}

\subsection{Descent direction framework}
\label{sec:directions}

We now explain the framework for choosing descent directions used later in the proof. The construction is most
naturally phrased at the level of Gram matrices.

Recall~\eqref{eq:defVSigma} and~\eqref{eq:VVWW}.
The tangent space to the manifold of rank-\(k\) positive semidefinite matrices at \(Y\),
restricted to \(\Sym(\cU)\), is
\[
    T_Y
    =
    \left\{
    VA V^\top+VBW^\top+WB^\top V^\top:
    A \in \Sym(k),\ B\in\R^{k\times m}
    \right\},
    \qquad
    m=p-k.
\]
Equivalently, \(T_Y\) is the set of first-order Gram perturbations
\begin{align}\label{eq:factorperturbation}
    \dot Y=Z\dot Z^\top+\dot Z Z^\top.
\end{align}
Indeed, if
\[
    \dot Y=VA V^\top+VBW^\top+WB^\top V^\top,
    \qquad A=A^\top,
\]
then it is realized by the ``factor perturbation''
\[
    \dot Z=
    \left(\frac12VA+WB^\top\right)\Sigma^{-1/2}.
\]

Define the \emph{measurement energy}
\begin{align}\tag{def-$Q$}
    Q(\dot Y):=\frac12\ip{\dot Y}{L\dot Y},
    \qquad
    \dot Y\in T_Y.
\end{align}
This is a positive definite quadratic form on \(T_Y\), and it is the second term in
the second-order criticality condition~\eqref{eq:2-criticality}.

Rather than choosing \(\dot Y\) directly, we choose a ``dual probe''
\(\alpha\in\R^{k\times m}\), and then let the measurement geometry choose the best
tangent vector $\dot Y_\alpha$ exposed by that probe. Concretely, given \(\alpha\in\R^{k\times m}\),
define the tangent vector
\begin{align}\label{eq:defGalpha}
    G_\alpha:=V\alpha W^\top+W\alpha^\top V^\top \in T_Y,
    \tag{def-$G_\alpha$}
\end{align}
and define the candidate descent direction \(\dot Y_\alpha\in T_Y\) by the \emph{variational problem}
\begin{align}\label{eq:dotYalpha}
    \dot Y_\alpha
    :=
    \argmax_{\dot Y\in T_Y}
    \left\{
        \ip{G_\alpha}{\dot Y}
        -
        Q(\dot Y)
    \right\}.
    \tag{def-$\dot Y_\alpha$}
\end{align}
Equivalently, \(\dot Y_\alpha\) realizes the Fenchel conjugate of \(Q\) at the linear functional \(\dot Y\mapsto \langle G_\alpha, \dot Y \rangle \):
\[
    Q^\ast(G_\alpha)
    :=
    \sup_{\dot Y\in T_Y}
    \left\{
        \ip{G_\alpha}{\dot Y}
        -
        Q(\dot Y)
    \right\}
    =
    Q(\dot Y_\alpha).
\]
Thus \(G_\alpha\) should be viewed as a hyperplane normal, or cotangent vector, while
\(\dot Y_\alpha\) is the tangent vector exposed by that hyperplane. We return to this
perspective below.

\begin{remark}[$\dot Y_\alpha$ has the optimal choice of \(V,V\) block]\label{rem:optimalVV}
There is another way to understand~\eqref{eq:dotYalpha}: the probe \(\alpha\) first selects
a cross-block \(B=B_\alpha\). Given this choice of \(B\), the \(V,V\) block \(A\) is then
chosen optimally, namely as the minimizer of the right-hand side of~\eqref{eq:2-criticality}.
Indeed, writing \(\dot Y=VAV^\top+VBW^\top+WB^\top V^\top\), $\dot Y_\alpha$ is the maximizer of
\[
\begin{aligned}
    \sup_{\dot Y\in T_Y}
    \left\{
        \ip{G_\alpha}{\dot Y}-Q(\dot Y)
    \right\}
    &=
    \sup_{B\in\R^{k\times m}}
    \left\{
        2\ip{\alpha}{B}
        -
        \inf_{A \in \Sym(k)}
        Q\!\left(VAV^\top+VBW^\top+WB^\top V^\top\right)
    \right\}.
\end{aligned}
\]
On the other hand, in~\eqref{eq:2-criticality}, the first term depends
only on \(B\) (see~\eqref{eq:onlyB} below), while the second term is exactly
\(Q(\dot Y)\). Hence, once \(B\) is fixed, the best chance of violating second-order
criticality is obtained by choosing \(A\) to minimize \(Q(\dot Y)\), exactly as in the nested optimization expression above. 
\end{remark}

\subsubsection*{Equivalent characterizations of \(\dot Y_\alpha\)}

The following lemma records several equivalent ways to define \(\dot Y_\alpha\).

\begin{lemma}
\label{lem:equiv-descent-direction}
For each \(\alpha\in\R^{k\times m}\), the vector \(\dot Y_\alpha\) is the unique element
of \(T_Y\) satisfying
\begin{align}\label{eq:Riesz}
    \ip{\dot Y}{L\dot Y_\alpha}
    =
    \ip{\dot Y}{G_\alpha}
    \qquad
    \forall \dot Y\in T_Y.
\end{align}
In particular, $Q(\dot Y_\alpha)
    =
    \frac12\ip{G_\alpha}{\dot Y_\alpha}.$

If \(\alpha\neq0\), then, up to multiplication by a nonzero scalar,
\(\dot Y_\alpha\) is the unique maximizer of each of the following optimization
problems:
\begin{align}
    \sup_{\dot Y\in T_Y,\ \dot Y\neq0}
    \Bigg\{\frac{\ip{G_\alpha}{\dot Y}^2}{Q(\dot Y)}\Bigg\}
    &=
    4Q(\dot Y_\alpha),
    \label{eq:Rayleigh_characterization}
    \\
    \sup_{\dot Y\in T_Y,\ Q(\dot Y)\le1}
    \Big\{\ip{G_\alpha}{\dot Y}\Big\}
    &=
    2\sqrt{Q(\dot Y_\alpha)}.
    \label{eq:support_G}
\end{align}
\end{lemma}

\begin{proof}
The maximizer in~\eqref{eq:dotYalpha} is unique because \(Q\) is positive definite on
\(T_Y\). The first-order optimality condition for~\eqref{eq:dotYalpha} is
\[
    0=
    \ip{G_\alpha}{\dot Y}
    -
    \ip{\dot Y}{L\dot Y_\alpha}
    \qquad
    \forall \dot Y\in T_Y,
\]
which proves~\eqref{eq:Riesz}. Taking \(\dot Y=\dot Y_\alpha\) in~\eqref{eq:Riesz}
gives $\ip{G_\alpha}{\dot Y_\alpha}
    =
    \ip{\dot Y_\alpha}{L\dot Y_\alpha}
    =
    2Q(\dot Y_\alpha).$

Finally, by Cauchy--Schwarz for the
inner product \((\dot Y_1,\dot Y_2)\mapsto\ip{\dot Y_1}{L\dot Y_2}\),
\[
    \ip{G_\alpha}{\dot Y}^2
    =
    \ip{\dot Y}{L\dot Y_\alpha}^2
    \le
    \ip{\dot Y}{L\dot Y}\,
    \ip{\dot Y_\alpha}{L\dot Y_\alpha}
    =
    4Q(\dot Y)Q(\dot Y_\alpha),
\]
with equality if and only if \(\dot Y\) is proportional to \(\dot Y_\alpha\).
This proves~\eqref{eq:Rayleigh_characterization}. Restricting to
\(Q(\dot Y)\le1\) gives $\ip{G_\alpha}{\dot Y}
    \le
    2\sqrt{Q(\dot Y_\alpha)},$
with equality at
\[
    \dot Y
    =
    \frac{\dot Y_\alpha}{\sqrt{Q(\dot Y_\alpha)}}.
\]
This proves~\eqref{eq:support_G}.
\end{proof}

\subsubsection*{A dual interpretation of second-order criticality via ellipsoid containment}

We now explain the geometric meaning of the construction above. 
Let
\begin{align}\label{eq:dotYAB}
    \dot Y=VAV^\top+VBW^\top+WB^\top V^\top\in T_Y,
\end{align}
and let $\dot Z=\left(\frac12VA+WB^\top\right)\Sigma^{-1/2}$ be the factor perturbation realizing \(\dot Y\).
Since \(S=W\Omega W^\top\) by Lemma~\ref{lem:stress-psd}, the first term in~\eqref{eq:2-criticality} satisfies
\begin{align}\label{eq:onlyB}
    \ip{\dot Z\dot Z^\top}{S}
    =
    \tr(B^\top\Sigma^{-1}B\Omega).
\end{align}
Thus, keeping in mind~\eqref{eq:dotYAB}, second-order criticality can be written as
\begin{align}\label{eq:2crit_RQ}
    \tr(B^\top\Sigma^{-1}B \Omega)
    \le
    Q(\dot Y)
    \qquad
    \forall \dot Y\in T_Y.
    \tag{primal-2-criticality}
\end{align}

Define two subsets of \(\R^{k\times m}\):
\begin{align*}
    \mathcal K_Q
    &:=
    \{B \in \R^{k\times m} : \ Q(VAV^\top+VBW^\top+WB^\top V^\top)\le1 \text{  for some $A \in \Sym(k)$}\},\\
    \mathcal K_\Omega
    &:=
    \{B\in\R^{k\times m}:
    \tr(B^\top\Sigma^{-1}B\Omega)\le1\}.
\end{align*}
Then~\eqref{eq:2-criticality} is equivalent, by~\eqref{eq:2crit_RQ} and homogeneity, to the containment
\begin{align}\label{eq:ellipsoid_containment}
    \mathcal K_Q\subseteq\mathcal K_\Omega.
    \tag{ellipsoids}
\end{align}
These sets are ellipsoids, possibly degenerate if \(\Omega\) is singular.\footnote{
Indeed,
\(\mathcal K_Q\) is the image under the projection
\((A,B)\mapsto B\)
of the ellipsoid $\{(A,B):Q(VAV^\top+VBW^\top+WB^\top V^\top)\le1\}.$
Equivalently, the function $q(B):=\min_{A}
Q(VAV^\top+VBW^\top+WB^\top V^\top)$
is a positive definite quadratic form on
\(\R^{k\times m}\), and
\(\mathcal K_Q=\{B:q(B)\le1\}\).
}

Recall that the \emph{support function} of a convex set \(\mathcal K\) is
\[
    h_{\mathcal K}(\alpha)
    :=
    \sup_{B\in\mathcal K}\ip{\alpha}{B}.
\]
For a fixed $\alpha$, a maximizer $B \in\mathcal K$ is called a \emph{support point} corresponding to $\alpha$.
Geometrically, \(h_{\mathcal K}(\alpha)\) is the farthest distance that \(\mathcal K\)
reaches in the direction \(\alpha\); the hyperplane
\[
    \{B:\ip{\alpha}{B}=h_{\mathcal K}(\alpha)\}
\]
supports \(\mathcal K\). A standard fact from convex duality~\citep[Ex.~3.35]{boyd2004convex} is that, for closed convex
sets,
\[
    \mathcal K_1\subseteq \mathcal K_2
    \qquad\Longleftrightarrow\qquad
    h_{\mathcal K_1}(\alpha)\le h_{\mathcal K_2}(\alpha)
    \quad\text{for all }\alpha.
\]
Equivalently, non-containment is certified by a single ``separating'' hyperplane $\alpha$ satisfying $h_{\mathcal K_1}(\alpha) > h_{\mathcal K_2}(\alpha)$.

In the present setting, the support function of \(\mathcal K_Q\) is exactly what
\(\dot Y_\alpha\) computes. Indeed, by~\eqref{eq:support_G} in Lemma~\ref{lem:equiv-descent-direction},
\[
h_{\mathcal K_Q}(\alpha)
=
\sup_{\dot Y\in T_Y,\ Q(\dot Y)\le1}
\ip{\alpha}{B}
=
\frac12
\sup_{\dot Y\in T_Y,\ Q(\dot Y)\le1}
\ip{G_\alpha}{\dot Y}
=
\sqrt{Q(\dot Y_\alpha)},
\]
where the supremum is attained at a positive multiple of \(\dot Y_\alpha\).

The support function of \(\mathcal K_\Omega\) is also explicit. With the change of variables
\(X=\Sigma^{-1/2}B\Omega^{1/2}\),
\[
\begin{aligned}
h_{\mathcal K_\Omega}(\alpha)
&=
\sup_{\tr(B^\top\Sigma^{-1}B\Omega)\le1}
\ip{\alpha}{B} \\
&=
\sup_{\|X\|_{\mathrm F}\le1}
\ip{\Sigma^{1/2}\alpha\Omega^{\dagger/2}}{X} \\
&=
\|\Sigma^{1/2}\alpha\Omega^{\dagger/2}\|_{\mathrm F} 
=
\sqrt{\tr(\alpha^\top\Sigma\alpha\,\Omega^\dagger)},
\end{aligned}
\]
where the second equality uses the assumption
\(\im\alpha^\top\subseteq\im\Omega\), and the third follows from
Cauchy--Schwarz.\footnote{When
\(\Omega\) is singular, the support function of \(\mathcal K_\Omega\) is finite precisely
when \(\im \alpha^\top \subseteq \im\Omega\).}

Therefore,~\eqref{eq:2-criticality} is equivalent to the support-function inequality
\begin{equation}\label{eq:support_2crit}\tag{dual-2-criticality}
\begin{split}
    Q(\dot Y_\alpha)
    =
    h_{\mathcal K_Q}(\alpha)^2
    &\le
    h_{\mathcal K_\Omega}(\alpha)^2
    =
    \tr(\alpha^\top\Sigma\alpha\,\Omega^\dagger) \\ &\forall \alpha \in \mathbb{R}^{k \times m} \text{ satisfying $\im \alpha^\top \subseteq \im \Omega$}.
\end{split}
\end{equation}
This is the dual version of second-order criticality. We derive the same inequality again,
by a direct algebraic calculation, in Lemma~\ref{lem:exact-rank-one} of Section~\ref{sec:second-order}.

How does the variationally-defined direction \(\dot Y_\alpha\) arise from this dual
perspective? The probe \(\alpha\) is a candidate separating hyperplane for the
containment~\eqref{eq:ellipsoid_containment}. If \(\alpha\) violates the support-function
inequality~\eqref{eq:support_2crit}, then the containment fails, and therefore there is a
tangent vector violating the primal second-order condition~\eqref{eq:2crit_RQ}. This
tangent vector is precisely \(\dot Y_\alpha\), up to harmless scaling. More explicitly, if
\(\alpha\neq0\), then up to normalization, the cross block
\[
    B_\alpha:=V^\top\dot Y_\alpha W
\]
is the support point of \(\mathcal K_Q\) exposed by \(\alpha\).

\begin{remark}[Comparison with \citet{criscitiello2025snl}]
\label{rem:comparisontoolddescentdirections}
This ellipsoid viewpoint also clarifies the relation between the descent directions used here and
those in \citep[Thm.~1.1]{criscitiello2025snl}. Fix a probe \(\alpha\). The support point of the ellipsoid \(\mathcal K_\Omega\) in direction \(\alpha\) is, up to normalization,
\[
    B_\Omega(\alpha)=\Sigma\alpha\Omega^\dagger.
\]
\citet{criscitiello2025snl} use the descent directions 
$\dot Y = V B_\Omega(\alpha) W^\top + W B_\Omega(\alpha)^\top V^\top$.
By contrast, our proof uses the support point of the ellipsoid
\(\mathcal K_Q\) in the same direction \(\alpha\). This support point is
\[
    B_Q(\alpha):=V^\top\dot Y_\alpha W,
\]
and the full tangent vector \(\dot Y_\alpha\) also chooses the \(V,V\) block optimally (see Remark~\ref{rem:optimalVV}).
Both proofs use the same dual probes, but expose different ellipsoids: the earlier
directions come from \(\mathcal K_\Omega\), whereas the directions here come from
\(\mathcal K_Q\). 
The latter is the sharper and more natural choice for a fixed probe.
\end{remark}

\subsection{Descent directions under the structured-inverse hypothesis}\label{sec:descentdirectionsunder}

The variational definition of \(\dot Y_\alpha\)~\eqref{eq:dotYalpha} is conceptually useful, but to show that
\(\dot Y_\alpha\) is a descent direction we require a characterization that is as explicit
as possible. In this subsection, we derive one by exploiting the structure of
\(L\) provided by~\eqref{eq:F1}--\eqref{eq:F3}. Moreover, in light of
\eqref{eq:support_2crit}, we also derive an equally explicit lower bound on \(Q(\dot Y_\alpha)\).

Recall~\eqref{eq:defviwi} and~\eqref{eq:defGalpha}, and define the vector
\begin{align}\label{eq:def_source}
    s_\alpha\in\R^N,
    \qquad \text{with entries} \qquad
    s_\alpha(i):= \frac{1}{2} a_i^\top G_\alpha a_i = v_i^\top\alpha w_i.
    \tag{def-$s_\alpha$}
\end{align}
Also define the symmetric matrix
\begin{align}\label{eq:PiW}
    \Pi_W\in\R^{N\times N},
    \qquad
    (\Pi_W)_{ij}:= (a_i^\top WW^\top a_j)^2 = (w_i^\top w_j)^2.
    \tag{def-$\Pi_W$}
\end{align}
When $L = \Delta^* \Delta$, $s_\alpha = \frac{1}{2} \diag(G_\alpha)$ and $\Pi_W = (WW^\top) \odot (WW^\top)$, where $\odot$ denotes the entrywise product.

\begin{lemma}[Explicit description of $\dot Y_\alpha$]
\label{lem:best-response}
The matrix \(I-\Pi_W\) is positive definite.  Moreover:
\begin{itemize}
\item Defining
\begin{align}\label{eq:defhalpha}
    h_\alpha :=(I-\Pi_W)^{-1}s_\alpha, \qquad \beta_\alpha:=2\sum_i h_\alpha(i)w_i^{} w_i^\top,
    \tag{def-$h_\alpha, \beta_\alpha$}
\end{align}
and recalling~\eqref{eq:dotYalpha}, we have
\begin{align*}
    \dot Y_\alpha
    &=
    L^{-1}(G_\alpha+W\beta_\alpha W^\top)\\
    &=
    G_\alpha+W\beta_\alpha W^\top
    -
    \Gamma(G_\alpha+W\beta_\alpha W^\top) \\
    &=
    G_\alpha+W\beta_\alpha W^\top
    -2 \sum_i h_\alpha(i) a_i^{} a_i^\top
\end{align*}
\item A factor perturbation~\eqref{eq:factorperturbation} realizing \(\dot Y_\alpha\) is $\dot Z_\alpha
    = \left(\frac12VA_\alpha+WB_\alpha^\top\right)\Sigma^{-1/2}$
for some $A_\alpha$, and
$$B_\alpha = V^\top\dot Y_\alpha W
=
\alpha
    -
    V^\top\Gamma(G_\alpha+W\beta_\alpha W^\top)W
    =
    \alpha-2\sum_i h_\alpha(i)v_i^{} w_i^\top.$$
\item Lastly, 
$Q(\dot Y_\alpha)
    =
    \ip{\alpha}{B_\alpha}
    =
    \fro{\alpha}^2-2s_\alpha^\top h_\alpha^{},$
where 
\begin{align*}
s_\alpha^\top h_\alpha^{}
    \le
    \sum_i\frac{s_\alpha(i)^2}{1-\tau_i}.
\end{align*}
\end{itemize}
\end{lemma}

\begin{proof}
We first prove that \(I-\Pi_W\) is positive definite using its Laplacian structure.
Define
\[
    \rho_i:=\sum_{j=1}^N (w_i^\top w_j)^2
    =
    w_i^\top\left(\sum_{j=1}^N w_jw_j^\top\right)w_i .
\]
Using~\eqref{eq:sumvvtwwt} and~\eqref{eq:viwinorm}, we have
\begin{align}\label{eq:thisguy}
    0\le \rho_i\le \|w_i\|^2=\tau_i \leq 1-\eta.
\end{align}
Now define
\[
    L_W:=\Diag(\rho_i)-\Pi_W.
\]
Then \(L_W\) is a weighted graph Laplacian, corresponding to the graph with weights $(\Pi_W)_{ij} = (w_i^\top w_j)^2 \ge0$. 
Therefore, $L_W \succeq 0$, and
\begin{align}\label{eq:fromproof}
    I-\Pi_W
    =
    \Diag(1-\rho_i)+L_W \succeq \Diag(1-\rho_i)\succeq \Diag(1-\tau_i)\succeq \eta I,
\end{align}
using~\eqref{eq:thisguy}.
Hence \(I-\Pi_W\succ0\).

\paragraph{Computing $\dot Y_\alpha$.}
Define $\hat Y_\alpha := L^{-1}(G_\alpha+W\beta_\alpha W^\top)$.
We aim to show that $\hat Y_\alpha$ is in $T_Y$ and satisfies the stationarity condition~\eqref{eq:Riesz} in Lemma~\ref{lem:equiv-descent-direction}.
Uniqueness then implies $\dot Y_\alpha = \hat Y_\alpha$.

Using~\eqref{eq:def_source},~\eqref{eq:PiW} and~\eqref{eq:defhalpha}, for each $i \in \{1, \ldots, N\}$ we have
\begin{align*}
    a_i^\top(G_\alpha+W\beta_\alpha W^\top)a_i
    &=
    2s_\alpha(i)+ w_i^\top \beta_\alpha w_i
    =
    2s_\alpha(i)+ 2 w_i^\top \bigg(\sum_j h_\alpha(j) w_j w_j^\top\bigg) w_i \\
    &=
    2s_\alpha(i)+ 2 \sum_j h_\alpha(j) (\Pi_W)_{ij}
    =
    2s_\alpha(i)+2(\Pi_Wh_\alpha)(i) \\
    &=
    2\big[(I - \Pi_W) h_\alpha+\Pi_Wh_\alpha\big](i)
    =
    2h_\alpha(i).
\end{align*}
Hence, by~\eqref{eq:F3},
\[
\begin{aligned}
    \hat Y_\alpha
    &=
    L^{-1}(G_\alpha+W\beta_\alpha W^\top)\\
    &=
    G_\alpha+W\beta_\alpha W^\top
    -
    \Gamma(G_\alpha+W\beta_\alpha W^\top)\\
    &=
    G_\alpha+W\beta_\alpha W^\top
    -
    \sum_i a_i^\top(G_\alpha+W\beta_\alpha W^\top) a_i \cdot a_i^{} a_i^\top\\
    &=
    G_\alpha+W\beta_\alpha W^\top-2\sum_i h_\alpha(i) a_i^{} a_i^\top.
\end{aligned}
\]
From this expression, we conclude $\hat Y_\alpha \in T_Y$.  Indeed, multiplying both sides by $W$:
\[
    W^\top\hat Y_\alpha W
    =
    \beta_\alpha-2\sum_i h_\alpha(i)w_iw_i^\top
    =
    0,
\]
using the definition~\eqref{eq:defhalpha}.

Moreover, $L\hat Y_\alpha=G_\alpha+W\beta_\alpha W^\top$ as an immediate consequence of the definition of \(\hat Y_\alpha\).
Therefore, for every \(\dot Y\in T_Y\), we have
\begin{align}\label{eq:stationarityagain}
    \ip{\dot Y}{L\hat Y_\alpha}
    =
    \ip{\dot Y}{G_\alpha}.
\end{align}
Appealing to Lemma~\ref{lem:equiv-descent-direction}, we conclude $\dot Y_\alpha = \hat Y_\alpha$.  This proves the first bullet.

\paragraph{The cross block $B_\alpha$.}
Using the formula for \(\dot Y_\alpha\) just derived,
\[
    B_\alpha = V^\top \dot Y_\alpha W
    = \alpha
    -
    V^\top\Gamma(G_\alpha+W\beta_\alpha W^\top)W
    =
    \alpha-2\sum_i h_\alpha(i)v_i^{} w_i^\top.
\]

\paragraph{Computing $Q(\dot Y_\alpha)$.}
By~\eqref{eq:stationarityagain} (or Lemma~\ref{lem:equiv-descent-direction}), the formula for $B_\alpha$ just derived, and~\eqref{eq:def_source},
\begin{align*}
    Q(\dot Y_\alpha) &=
    \frac12\ip{\dot Y_\alpha}{L\dot Y_\alpha}
    =
    \frac12\ip{\dot Y_\alpha}{G_\alpha}
    =
    \ip{\alpha}{B_\alpha} \\
    &= \fro{\alpha}^2
    -
    2\sum_i h_\alpha(i)v_i^\top\alpha w_i
    =
    \fro{\alpha}^2-2s_\alpha^\top h_\alpha^{}.
\end{align*}
Let us bound the last term $s_\alpha^\top h_\alpha^{}$.
As shown in equation~\eqref{eq:fromproof}, $I-\Pi_W
    \succeq
    \Diag(1-\tau_i).$
Since both sides are positive definite, $(I-\Pi_W)^{-1}
    \preceq
    \Diag\left(\frac1{1-\tau_i}\right).$
In particular,
\begin{equation*}
s_\alpha^\top h_\alpha = s_\alpha^\top (I-\Pi_W)^{-1}s_\alpha \leq s_\alpha^\top \Diag\left(\frac1{1-\tau_i}\right) s_\alpha = \sum_i\frac{s_\alpha(i)^2}{1-\tau_i}.
\qedhere
\end{equation*}
\end{proof}

\subsection{Consequences of second-order criticality}
\label{sec:second-order}

We now prove Theorem~\ref{thm:intermediate} using the descent directions developed in
Sections~\ref{sec:directions} and~\ref{sec:descentdirectionsunder}. The argument
proceeds in three steps. First, we specialize the dual probe to rank-one matrices
\(\alpha=u\xi^\top\). Second, we average over \(u\in\ker M=\ker(V^\top Y_\star V)\)
and \(\xi\in\im\Omega\) with covariance \(R\). Finally, we choose
\(R=K=W^\top Y_\star W\), from which the theorem follows.

Throughout this subsection, \(Z\) denotes a full-rank second-order critical point, and
we adopt the notation of Section~\ref{sec:first-order}. By
Lemma~\ref{lem:stress-psd}, $S=W\Omega W^\top,
\Omega\succeq0,$
and we recall the matrices
\[
M:=V^\top Y_\star V=\Sigma-P,
\qquad
P:=V^\top\Gamma(S)V,
\qquad
K:=W^\top Y_\star W=\Omega-W^\top\Gamma(S)W,
\]
defined in~\eqref{eq:M}, \eqref{eq:defP}, and~\eqref{eq:K}.

\begin{lemma}[Rank-one $\alpha$]
\label{lem:exact-rank-one}
For every \(u\in\R^k\), every \(\xi\in\im\Omega\), and
\(\alpha=u\xi^\top\),
\begin{align}\label{eq:rank-one-2crit}
    0
    \le
    (\xi^\top\Omega^\dagger \xi)\,u^\top(M+P)u
    -
    \norm{\xi}^2\norm{u}^2
    +
    2s_\alpha^\top h_\alpha^{}.
\end{align}
\end{lemma}

\begin{proof}
We give two proofs, both using the descent direction \(\dot Y=\dot Y_\alpha\) with \(\alpha=u\xi^\top\). The first relies on the support-function formulation~\eqref{eq:support_2crit}, whereas the second derives the same estimate directly from the second-order condition~\eqref{eq:2crit_RQ}.

\paragraph{Proof 1.} Since \(\xi\in\im\Omega\), \(\alpha=u\xi^\top\) satisfies
\(\im\alpha^\top\subseteq\im\Omega\). Hence~\eqref{eq:support_2crit} gives
\[
    Q(\dot Y_\alpha)
    \le
    \tr(\alpha^\top\Sigma\alpha\,\Omega^\dagger)
    =
    (\xi^\top\Omega^\dagger\xi)\,u^\top\Sigma u .
\]
On the other hand, the third bullet of Lemma~\ref{lem:best-response} gives
\[
    Q(\dot Y_\alpha)
    =
    \fro{\alpha}^2-2s_\alpha^\top h_\alpha
    =
    \norm{u}^2\norm{\xi}^2-2s_\alpha^\top h_\alpha .
\]
Using \(\Sigma=M+P\), which follows from~\eqref{eq:M}, and rearranging gives
\eqref{eq:rank-one-2crit}.

\paragraph{Proof 2.}
By~\eqref{eq:2crit_RQ} with $\dot Y = \dot Y_\alpha$,
\begin{align}\label{eq:ineq1}
    \tr(B_\alpha^\top\Sigma^{-1}B_\alpha\Omega)
    \le
    Q(\dot Y_\alpha),
\end{align}
where \(B_\alpha=V^\top\dot Y_\alpha W\).
Since \(\xi\in\im\Omega\),
\(
\xi\xi^\top\preceq(\xi^\top\Omega^\dagger\xi)\Omega
\)
by Cauchy--Schwarz, and therefore
\begin{align}\label{eq:ineq2}
    \xi^\top B_\alpha^\top\Sigma^{-1}B_\alpha\xi
    \le
    (\xi^\top\Omega^\dagger\xi)\,
    \tr(B_\alpha^\top\Sigma^{-1}B_\alpha\Omega).
\end{align}
Using $Q(\dot Y_\alpha) = \langle u \xi^\top, B_\alpha\rangle$ by Lemma~\ref{lem:best-response}, and applying Cauchy--Schwarz in the \(\Sigma\)-inner product,
\begin{align}\label{eq:ineq3}
    Q(\dot Y_\alpha)^2
    =
    (u^\top B_\alpha\xi)^2
    \le
    (u^\top\Sigma u)\,
    \xi^\top B_\alpha^\top\Sigma^{-1}B_\alpha\xi.
\end{align}
Combining the inequalities~\eqref{eq:ineq1},~\eqref{eq:ineq2} and~\eqref{eq:ineq3} gives
\[
    Q(\dot Y_\alpha)^2
    \le
    (\xi^\top\Omega^\dagger\xi)\,
    (u^\top\Sigma u)\,
    Q(\dot Y_\alpha).
\]
Hence
\[
    Q(\dot Y_\alpha)
    \le
    (\xi^\top\Omega^\dagger\xi)\,
    u^\top\Sigma u,
\]
the case \(Q(\dot Y_\alpha)=0\) being trivial.
The conclusion now follows exactly as in the first proof.
\end{proof}

\subsubsection*{Averaging over $u, \xi$} 
We now restrict \(u\) to the kernel of the \(V,V\) block of $Y_\star$. Define
\begin{align}\label{eq:defNspace}\tag{def-$\mathcal{N}, d$}
    \mathcal N:=\ker M,
    \qquad
    d:=\dim\mathcal N.
\end{align}
Since \(M=V^\top Y_\star V = V^\top Z_\star^{} Z_\star^\top V\) has rank at most $\ell$,
\begin{align}\label{eq:dimNlower}
    d
    =
    k-\rank(M)
    \ge
    k-\ell.
\end{align}
The subspace \(\mathcal N\) is where overparameterization enters the proof: increasing
\(k\) increases the guaranteed dimension of \(\mathcal N\).

\begin{lemma}
\label{lem:traced-covariance}
For every \(R\succeq0\) with \(\im R\subseteq\im\Omega\),
\begin{align}\label{eq:traced-covariance-rough}
    \tr(R)d
    \le
    (1-\eta)\big[\tr(\Omega^\dagger R)\,\tr(K)
    +
    2\tr(R)\big].
\end{align}
\end{lemma}

\begin{proof}
Fix \(\xi\in\im\Omega\) and \(u\in\mathcal N=\ker M\), and set
\(\alpha=u\xi^\top\). Lemma~\ref{lem:exact-rank-one} gives
\[
    0
    \le
    (\xi^\top\Omega^\dagger \xi)\,u^\top P u
    -
    \norm{\xi}^2\norm{u}^2
    +
    2s_\alpha^\top h_\alpha .
\]
By Lemma~\ref{lem:best-response} and~\eqref{eq:def_source},
\[
    s_\alpha^\top h_\alpha^{}
    \le
    \sum_i\frac{s_\alpha(i)^2}{1-\tau_i}
    =
    u^\top
    \left(
        \sum_i
        \frac{w_i^\top \xi\xi^\top w_i}{1-\tau_i}
        v_i^{} v_i^\top
    \right)u .
\]
Thus
\begin{align}\label{eq:precedinginequ}
    0
    \le
    \tr(\Omega^\dagger \xi\xi^\top)\,u^\top P u
    -
    \tr(\xi\xi^\top)\norm{u}^2
    +
    2u^\top
    \left(
        \sum_i
        \frac{w_i^\top \xi\xi^\top w_i}{1-\tau_i}
        v_i^{} v_i^\top
    \right)u .
\end{align}
Since \(R\succeq0\) and \(\im R\subseteq\im\Omega\), it admits a decomposition
\(R=\sum_j \xi_j^{} \xi_j^\top\) with \(\xi_j\in\im\Omega\). Applying
\eqref{eq:precedinginequ} to each \(\xi_j\) and summing over $j$ gives, for every
\(u\in\mathcal N\),
\begin{align}\label{eq:pointwise-covariance}
    0
    \le
    u^\top
    \left[
        \tr(\Omega^\dagger R)P
        -
        \tr(R)I_k
        +
        2\sum_i
        \frac{w_i^\top Rw_i}{1-\tau_i}
        v_i^{} v_i^\top
    \right]u .
\end{align}
Equivalently, the compression of the bracketed matrix to \(\mathcal N\) is positive
semidefinite. Taking its trace gives
\begin{align}\label{eq:traced-covariance-prelim}
    0
    \le
    \tr(\Omega^\dagger R)\tr_{\mathcal N}(P)
    - \tr(R) \tr_{\mathcal N}(I_k)
    +
    2\tr_{\mathcal N}
    \left(
        \sum_i
        \frac{w_i^\top Rw_i}{1-\tau_i}
        v_i^{} v_i^\top
    \right).
\end{align}
For the middle term, $\tr_{\mathcal N}(I_k) = \dim \mathcal{N} = d$.
We now bound the first and last terms. 

Since \(P\succeq0\), Lemma~\ref{lem:first-order-blocks}
implies
\begin{align}\label{eq:traceN-P-bound}
    \tr_{\mathcal N}(P)\le \tr(P)\le (1-\eta)\tr(K).
\end{align}
Moreover, using~\eqref{eq:viwinorm} and then~\eqref{eq:sumvvtwwt},
\begin{align}\label{eq:traceN-error-bound}
    \tr_{\mathcal N}
    \left(
        \sum_i
        \frac{w_i^\top Rw_i}{1-\tau_i}
        v_i v_i^\top
    \right)
    &\le
    \sum_i
    \frac{w_i^\top Rw_i}{1-\tau_i}\norm{v_i}^2 \notag\\
    &\le
    \sum_i
    w_i^\top Rw_i\,
    \frac{1-\tau_i-\eta}{1-\tau_i} \notag\\
    &\le
    (1-\eta)\sum_i w_i^\top Rw_i \notag\\
    &=
    (1-\eta)\left\langle \sum_i w_iw_i^\top,R\right\rangle
    \le
    (1-\eta)\tr(R).
\end{align}
Substituting~\eqref{eq:traceN-P-bound} and~\eqref{eq:traceN-error-bound}
into~\eqref{eq:traced-covariance-prelim} and rearranging yields~\eqref{eq:traced-covariance-rough}.
\end{proof}

\begin{remark}[$R$ as a covariance]
The matrix \(R\) may be viewed as a covariance. Indeed, applying
Lemma~\ref{lem:exact-rank-one} with \(\alpha=u\xi^\top\), and then averaging over a
random vector \(\xi\in\im\Omega\) with covariance \(R\), gives precisely inequality~\eqref{eq:pointwise-covariance}. 
\end{remark}

\subsubsection*{Choosing the covariance $R$, and the proof of Theorem~\ref{thm:intermediate}}

We now choose \(R=K=W^\top Y_\star W\).
This is admissible because Lemma~\ref{lem:first-order-blocks} gives
\(0\preceq K\preceq\Omega\), and therefore
\(\im K\subseteq\im\Omega\).
We distinguish the two cases appearing in Theorem~\ref{thm:intermediate}.

\paragraph{Case \(K=0\).}
If \(K=0\), then \(Z\) is globally optimal. Indeed,
\[
    g(Z;Z_\star)
    =
    \ip{Y-Y_\star}{L(Y-Y_\star)}
    =
    \ip{Y_\star-Y}{S}
    =
    \ip{Y_\star-Y}{W\Omega W^\top}
    =
    \tr(K\Omega),
\]
where we used \(S=W\Omega W^\top\) (Lemma~\ref{lem:stress-psd}) and
\(W^\top(Y_\star-Y)W=W^\top Y_\star W=K\).
Hence \(K=0\) implies \(g(Z;Z_\star)=0\).

\paragraph{Case $K \neq 0$.}
Assume \(K\neq0\), which implies \(\tr K>0\). Applying
Lemma~\ref{lem:traced-covariance} with \(R=K\),
\[
    \tr(K)d
    \le
    (1-\eta)\bigl(\tr(\Omega^\dagger K)+2\bigr)\tr(K).
\]
Dividing by \(\tr K>0\),
\begin{align}\label{eq:d-bound-kappa}
    d\le
    (1-\eta)\bigl(\tr(\Omega^\dagger K)+2\bigr).
\end{align}

It remains to show that
\(\tr(\Omega^\dagger K)\le\rank K\).
Since
\(0\preceq K\preceq\Omega\),
\[
    0\preceq \Omega^{\dagger/2}K\Omega^{\dagger/2}
    \preceq
    P_{\im\Omega}.
\]
Thus all eigenvalues of \(\Omega^{\dagger/2}K\Omega^{\dagger/2}\) lie in \([0,1]\).
Therefore,
\[
    \tr(\Omega^\dagger K)
    =
    \tr(\Omega^{\dagger/2}K\Omega^{\dagger/2})
    \le
    \rank(\Omega^{\dagger/2}K\Omega^{\dagger/2})
    \le
    \rank K.
\]
Substituting this into~\eqref{eq:d-bound-kappa} and combining with~\eqref{eq:dimNlower} gives
\begin{align}\label{eq:d-bound-rankK}
    k-\ell
    \le d
    \le
    (1-\eta)(\rank K+2)
    =
    (1-\eta)\bigl(\rank(W^\top Y_\star W)+2\bigr),
\end{align}
which is exactly the claimed alternative in Theorem~\ref{thm:intermediate}.

\begin{remark}[Randomized descent directions]
Lemma~\ref{lem:traced-covariance} also admits a probabilistic proof. Let
\(U_{\mathcal N}\in\R^{k\times d}\) have orthonormal columns spanning
\(\mathcal N=\ker M\), let \(R\succeq0\) satisfy
\(\im R\subseteq\im\Omega\), choose a factorization $R=FF^\top,$
and let \(G\in\R^{d\times q}\) have independent standard Gaussian entries. Define the
random probe
\[
    \alpha:=U_{\mathcal N}GF^\top.
\]
Applying~\eqref{eq:support_2crit} to this
random \(\alpha\), taking expectations, and estimating the resulting terms exactly as in
the proof of Lemma~\ref{lem:traced-covariance}, recovers
\eqref{eq:traced-covariance-rough}.

For the covariance used in the proof, $R=K=(W^\top Z_\star)(W^\top Z_\star)^\top,$
we may take \(F=W^\top Z_\star\). Writing $Z_{\rm ls}:=WW^\top Z_\star$ (Remark~\ref{rem:lsresidual}),
the corresponding random dual tangent vector is
\[
    G_\alpha
    =
    (VU_{\mathcal N}G)Z_{\rm ls}^\top
    +
    Z_{\rm ls}(VU_{\mathcal N}G)^\top.
\]
Thus the proof may equivalently be viewed as averaging over random dual probes which couple
random motions in the kernel \(\mathcal N\) with the least-squares residual
\(Z_{\rm ls}\). The associated descent directions \(\dot Y_\alpha\) are then random as
well.
\end{remark}

\section{The codimension-one case}
\label{sec:meq1}

We now consider the case
\[
    m:=p-k=1.
\]
At a full-rank point $Z$, \(m\) is the codimension of \(\im(Z)\) within \(\mathcal U\).
For complete-graph EDG, \(p=n-1\), so \(m=1\) is equivalent to
\[
    k=n-2.
\]
Thus, from a computational point of view, this regime is not especially attractive:
the optimization dimension \(k\) scales with \(n\), rather than satisfying \(k\ll n\).

We nevertheless study this case carefully for two reasons. First, a refined analysis
allows us to reach the conjectured threshold \(k\ge \ell+1\); see
Theorem~\ref{thm:meq1-complete}. For complete-graph EDG, this provides
theoretical evidence, complementing existing numerical evidence, for the conjecture that relaxing by a single dimension
is sufficient to obtain a benign landscape.

Second, the proof introduces a new mechanism that we expect may be useful in
approaching the full \(k\ge \ell+1\) conjecture. The argument of
Section~\ref{sec:proof-intermediate} exploits test directions indexed by vectors in
\(\ker M\). At the endpoint \(\dim\ker M=1\), however, these kernel directions alone
are no longer sufficient. 
We show that the vanishing of the generalized Schur complement of
\(Y_\star\) produces a canonical additional direction in
\((\ker M)^\perp\), which we call the \emph{Schur-companion direction}.
We expect that suitable higher-dimensional
analogues of these companion directions will play a role in proving the full
\(k\ge\ell+1\) conjecture.

Our main result in this section is the following.

\begin{theorem}
\label{thm:meq1-complete}
Assume \(L\) satisfies~\eqref{eq:F1}--\eqref{eq:F3}.
If \(m=p-k=1\) and
\[
    k\ge \ell+1,
\]
then every second-order critical point of~\eqref{eq:P} is globally optimal.
\end{theorem}

For complete-graph EDG, Theorem~\ref{thm:meq1-complete} applies when
\(k=n-2\). This has a useful small-\(n\) interpretation. Fix \(k=\ell+1\). When \(n\le\ell+2\), we have
\(k\ge n-1=p\), so benignness already follows from the convex regime of
Lemma~\ref{lem:kgeqp}. The first case beyond this convex regime is therefore $n=\ell+3$.
This is precisely the case \(m=1\), since $1 = m = n - 1 - k = n - 2 - \ell$.  Hence,
Theorem~\ref{thm:meq1-complete} proves benignness at the conjectured threshold $k = \ell + 1$
in the first nonconvex case not covered by Lemma~\ref{lem:kgeqp}, namely $n = \ell + 3$.

The proof of Theorem~\ref{thm:meq1-complete} is completed at the end of
Section~\ref{sec:meq1-frame-lemma}.

\subsection{Block notation and 2-criticality}
\label{sec:meq1-blocks}

We begin by introducing notation valid for arbitrary
\(
m=p-k.
\)
Let \(Z\) be a full-rank non-global second-order critical point of~\eqref{eq:P},
with \(\im(Z)\subseteq\mathcal U\). By Lemma~\ref{lem:stress-psd},
\[
    S=L(Y_\star-Y)
    =
    W\Omega W^\top,
    \qquad
    \Omega\succeq0.
\]
For each atom \(a_i\), \(i=1,\ldots,N\), recall we define
\[
    v_i:=V^\top a_i\in\R^k,
    \qquad
    w_i:=W^\top a_i\in\R^m,
\]
as in~\eqref{eq:defviwi}. 
Also, recall we define
\begin{align*} 
P&:=V^\top\Gamma(S)V=\sum_i d_i v_i^{} v_i^\top, \qquad \qquad d_i:=w_i^\top\Omega w_i \\ K&:=W^\top Y_\star W = \Omega-\sum_i d_i w_i^{} w_i^\top, \qquad M:=V^\top Y_\star V=\Sigma-P. 
\end{align*}
Taking the \(V,W\)-block of the identity
\(Y_\star=Y+S-\Gamma(S)\) from~\eqref{eq:exposingYstar} yields
\begin{align}
    C
    :=V^\top Y_\star W
    =
    -V^\top\Gamma(S)W
    =
    -\sum_i d_i v_i w_i^\top.
    \tag{$V,W$-block}
    \label{eq:meq1-crossblock-general}
\end{align}
Consequently,
\[
    [V,W]^\top Y_\star [V,W]
    =
    \begin{bmatrix}
        M & C\\
        C^\top & K
    \end{bmatrix}.
\]

Since \(Y_\star\succeq0\), the generalized Schur complement
\begin{align}
    K_{\rm com}
    :=
    K-C^\top M^\dagger C
    \tag{Schur compl}
    \label{eq:meq1-Kres-general}
\end{align}
is positive semidefinite~\citep[App.~A.5.5]{boyd2004convex}, and satisfies
\begin{align}\label{eq:schur-rank}
    \rank(Y_\star)
    =
    \rank(M)
    +
    \rank(K_{\rm com}).
    \tag{Schur rank}
\end{align}
Furthermore, if \(u\in\ker M\), then positivity of \(Y_\star\) implies
\begin{align}
    C^\top u=0.
    \label{eq:meq1-cross-kernel}
\end{align}

\subsubsection*{2-criticality and Woodbury's identity}

For
general \(m\), define the linear operator
\[
    \mathcal D_W:\Sym(m)\to\R^N,
    \qquad
    (\mathcal D_WH)_i:= a_i^\top W H W^\top a_i = w_i^\top H w_i,
\]
whose adjoint is
\[
    \mathcal D_W^\ast h
    :=
    \sum_i h_i w_iw_i^\top.
\]
Then
\(
\Pi_W=\mathcal D_W^{} \mathcal D_W^\ast,
\)
and, by Woodbury's identity~\citep[\S0.7.4]{Horn1994},
\begin{align}
    (I-\Pi_W)^{-1}
    =
    I
    +
    \mathcal D_W
    (I-\mathcal D_W^\ast\mathcal D_W^{})^{-1}
    \mathcal D_W^\ast.
    \tag{Woodbury}
    \label{eq:meq1-woodbury}
\end{align}
The inverse exists by Lemma~\ref{lem:best-response}.

By Lemma~\ref{lem:exact-rank-one}, second-order criticality implies that for every \(u\in\R^k\) and every
\(\xi\in\im\Omega\),
\begin{align}\label{eq:meq1-rank-one-before-special}
    0
    \le
    (\xi^\top\Omega^\dagger \xi)\,u^\top(M+P)u
    -
    \norm{\xi}^2\norm{u}^2
    +
    2s_\alpha^\top (I-\Pi_W)^{-1}s_\alpha \qquad \text{where} \qquad \alpha=u\xi^\top.
\end{align}
Using~\eqref{eq:meq1-woodbury}, the final term of the right-hand side can be written as
\begin{align}\label{eq:meq1-rank-one-before-special-2}
    2\norm{s_\alpha}^2
    +
    2s_\alpha^\top
    \mathcal D_W\,(I-\mathcal D_W^\ast\mathcal D_W)^{-1}\mathcal D_W^\ast s_\alpha.
\end{align}
In our earlier analysis, this final term was bounded by
\(
\sum_i s_\alpha(i)^2/(1-\tau_i).
\)
When \(m=1\), however, the operator
\(
I - \mathcal D_W^\ast\mathcal D_W^{}
\)
is simply multiplication by a scalar, which allows us to evaluate the Woodbury correction exactly rather than merely bounding it. This refinement is a key simplification in the \(m=1\) case.

\subsubsection*{Specialization to $m=1$}

We now specialize~\eqref{eq:meq1-rank-one-before-special} and~\eqref{eq:meq1-rank-one-before-special-2} to \(m=1\). 
In this case, we take $\xi = 1$, so that $\alpha = u \xi^\top = u$.
Additionally, \(\Omega\) is a nonnegative scalar, \(W\) is a unit vector in \(\cU\), and each
\(
    w_i:=W^\top a_i
\)
is a scalar. Since we are considering a non-global point $Z$, we know \(\Omega>0\); otherwise
\(S=0\), and hence \(Y=Y_\star\).

Define the following matrix, vector and scalar, which depend only on $\{a_i\}_i, V$ and $W$:
\begin{align}\label{eq:meq1-def-ptilde-r-q}
    \widetilde P:=\sum_i w_i^2v_i^{} v_i^\top \in \mathrm{PSD}(k),
    \qquad
    r:=\sum_i w_i^3 v_i \in \mathbb{R}^k,
    \qquad
    q:=1-\sum_i w_i^4 \in \mathbb{R}.
    \tag{def-$\tilde P, r, q$}
\end{align}
The scalar \(q\) is positive. Indeed, by~\eqref{eq:F1} and~\eqref{eq:F2},
\[
    \sum_i w_i^2
    =
    W^\top\left(\sum_i a_i a_i^\top\right)W
    \le1,
    \qquad
    \max_i w_i^2\le 1-\eta,
\]
so
\[
    \sum_i w_i^4
    \le
    \left(\max_i w_i^2\right)\sum_i w_i^2
    \le
    1-\eta<1.
\]
In the notation~\eqref{eq:meq1-def-ptilde-r-q}, the block identities become
\begin{align}\label{eq:meq1-blocks}
    M=\Sigma-\Omega\widetilde P,
    \qquad
    P=\Omega\widetilde P,
    \qquad
    C=-\Omega r,
    \qquad
    K=\Omega q.
\end{align}

Since $m=1$ and \(\xi=1\), we have
\(
    s_\alpha(i)=(v_i^\top u)w_i,
\)
and therefore
\[
    \mathcal D_W^\ast s_\alpha
    =
    \sum_i s_\alpha(i) w_i^2
    =
    \sum_i (v_i^\top u)w_i^3
    =
    u^\top r.
\]
Moreover, \(\mathcal D_W^\ast\mathcal D_W^{}\) is multiplication by
\(
    \sum_i w_i^4=1-q.
\)
Thus
\[
    (I-\mathcal D_W^\ast\mathcal D_W)^{-1}=\frac1q.
\]
Finally,
\[
    \norm{s_\alpha}^2
    =
    \sum_i (v_i^\top u)^2 w_i^2
    =
    u^\top\widetilde P u.
\]
Substituting these identities into~\eqref{eq:meq1-rank-one-before-special} and~\eqref{eq:meq1-rank-one-before-special-2} yields the following consequence of 2-criticality.

\begin{lemma}
\label{lem:m1-master}
In the \(m=1\) setting above,
\begin{align}\label{eq:m1-master}
    u^\top(I-3\widetilde P)u
    \le
    \frac1\Omega u^\top Mu
    +
    2\frac{(u^\top r)^2}{q}
    \qquad
    \forall u\in\R^k.
\end{align}
\end{lemma}

\subsection{Schur-companion directions}
\label{sec:schur-companion}

We next derive a key consequence of second-order criticality in the case \(m=1\). The
proof combines two complementary families of test directions $u$. The first consists of
vectors in \(\ker M\), which were already exploited in Section~\ref{sec:proof-intermediate}. The
second is a distinguished vector orthogonal to \(\ker M\), forced by the vanishing of the generalized Schur complement of \(Y_\star\). We refer to this vector as the \emph{Schur companion
direction}.

\begin{lemma}
\label{lem:endpoint-plane}
Suppose \(m=1\), \(k\ge\ell+1\), and \(Z\) is a full-rank non-global second-order
critical point. Then there exists a two-dimensional subspace
\(\mathcal F\subseteq\R^k\) such that
\begin{align}\label{eq:endpoint-plane-lower}
    \tr_{\mathcal F}(\widetilde P)
    +
    \frac{\|\Proj_{\mathcal F}r\|^2}{q}
    \ge
    \frac23,
    \tag{companion plane}
\end{align}
where \(\Proj_{\mathcal F} \colon \mathbb{R}^k \to \mathcal F\) denotes orthogonal projection onto \(\mathcal F\).
\end{lemma}

\begin{proof}
Recalling~\eqref{eq:defNspace} and~\eqref{eq:dimNlower}, we have $\mathcal N:=\ker M$, $d:=\dim\mathcal N$, and $d\ge k-\ell \ge 1.$

We distinguish two cases according to the nullity of \(M\).
The first case is straightforward: when \(\dim\ker M\ge2\), we simply choose
\(\mathcal F\subseteq\ker M\). The second case is more delicate. In that case
\(\ker M\) is one-dimensional, so we augment it with the Schur companion
direction \(u_\perp\) to obtain the required two-dimensional subspace.

\medskip

\noindent
\textbf{Case 1: \(d\ge2\).}
For every
\(u\in\mathcal N\), equation~\eqref{eq:meq1-cross-kernel} together with
\(C=-\Omega r\) from~\eqref{eq:meq1-blocks} implies
\(
    u^\top r=0.
\)
Inequality~\eqref{eq:m1-master} therefore reduces to
\[
    u^\top(I-3\widetilde P)u\le0 \qquad \forall u \in \mathcal N.
\]
Let \(\mathcal F\subseteq\mathcal N\) be any two-dimensional subspace.
Summing over an orthonormal basis of \(\mathcal F\) yields
\[
    \tr_{\mathcal F}(I-3\widetilde P)\le0.
\]
Since \(\dim\mathcal F=2\), this is equivalent to
\[
    2-3\tr_{\mathcal F}(\widetilde P)\le0,
\]
or equivalently,
\(
    \tr_{\mathcal F}(\widetilde P)\ge\frac23.
\)
Thus~\eqref{eq:endpoint-plane-lower} holds.

\medskip

\noindent
\textbf{Case 2: \(d=1\).}
Then \(d = k - \rank(M) \ge k-\ell \geq 1\) forces
\[
    k=\ell+1, \qquad \rank M = \ell.
\]
By~\eqref{eq:schur-rank},
$$\rank(K_{\rm com}) = \rank(Y_\star) - \rank(M) \leq \ell - \ell = 0.$$
Therefore, the generalized Schur complement --- which is a scalar since $m=1$ --- must vanish:
\begin{align}\label{eq:meq1-scalar-schur-vanish}
    K_{\rm com} = K-C^\top M^\dagger C=0.
\end{align}

Define the \emph{Schur companion direction}
\[
    u_\perp:=M^\dagger C \in \mathbb{R}^k.
    \tag{Schur companion}
\]
Note that \(u_\perp\in\im M=(\ker M)^\perp\).
Moreover, the identities~\eqref{eq:meq1-scalar-schur-vanish} and~\eqref{eq:meq1-blocks} give
\[
    u_\perp^\top M u_\perp
    =
    C^\top M^\dagger C
    =
    K
    =
    \Omega q.
\]
Since \(C=-\Omega r\), we also obtain
\[
    u_\perp^\top r=-q.
\]
Since \(q>0\), this implies \(u_\perp\neq0\).

Applying~\eqref{eq:m1-master} to \(u_\perp\) yields
\begin{align}\label{eq:thisguyatwo}
    u_\perp^\top(I-3\widetilde P)u_\perp
    &\le
    \frac1\Omega u_\perp^\top M u_\perp
    +
    2\frac{(u_\perp^\top r)^2}{q} =
    q+2q
    =
    3q.
\end{align}
Let \(u\) be a unit vector spanning \(\ker M\), and define
\[
    \mathcal F:=\operatorname{span}\{u,u_\perp\}.
\]
Since \(u\) and \(u_\perp\) are orthogonal,
\[
    \tr_{\mathcal F}(I-3\widetilde P)
    =
    u^\top(I-3\widetilde P)u
    +
    \frac{
        u_\perp^\top(I-3\widetilde P)u_\perp
    }{
        \|u_\perp\|^2
    }.
\]
The first term is nonpositive as argued in \textbf{Case 1}, while the second is bounded by
\(3q/\|u_\perp\|^2\), due to~\eqref{eq:thisguyatwo}. Hence
\begin{align}\label{eq:dodododo}
    \tr_{\mathcal F}(I-3\widetilde P)
    \le
    \frac{3q}{\|u_\perp\|^2}.
\end{align}

By Cauchy--Schwarz and \(u_\perp^\top r=-q\),
\[
    q^2 = (u_\perp^\top r)^2
    \le
    \|u_\perp\|^2
    \|\Proj_{\mathcal F}r\|^2,
\]
or equivalently,
\(
    \frac{3q}{\|u_\perp\|^2}
    \le
    3\frac{\|\Proj_{\mathcal F}r\|^2}{q}.
\)
Combining this with~\eqref{eq:dodododo} gives
\[
    \tr_{\mathcal F}(I-3\widetilde P)
    \le
    3\frac{\|\Proj_{\mathcal F}r\|^2}{q}.
\]
Since \(\dim\mathcal F=2\),
\[
    2-3\tr_{\mathcal F}(\widetilde P)
    \le
    3\frac{\|\Proj_{\mathcal F}r\|^2}{q},
\]
which rearranges to~\eqref{eq:endpoint-plane-lower}.
\end{proof}

\subsection{The frame obstruction}
\label{sec:meq1-frame-lemma}

We now show that the companion plane produced by Lemma~\ref{lem:endpoint-plane}
cannot exist, thereby showing that $Z$ cannot be simultaneously full-rank, non-global and second-order critical if $m=1$ and $k \geq \ell + 1$.
The key ingredient is the following purely geometric lemma, which depends only
on the frame atoms and the decomposition \([V,W]\), and not on first- or
second-order criticality.

\begin{lemma}
\label{lem:m1-coordinate}
Let $\cU\subseteq\R^n$ be a $p$-dimensional subspace, let $k \geq 2$, and assume $m = p - k = 1$.
Assume \([V,W]\) is an orthonormal basis of \(\cU\), where $V \in \mathbb{R}^{n \times k}, W \in \mathbb{R}^n$.
Also assume $a_1, \ldots, a_N \in \cU$ satisfy~\eqref{eq:F1}--\eqref{eq:F2}.
Define $\widetilde P, r, q$ as in~\eqref{eq:meq1-def-ptilde-r-q}.

For every two-dimensional subspace \(\mathcal F\subseteq\R^k\),
\begin{align}\label{eq:meq1-two-plane}
    \tr_{\mathcal F}(\widetilde P)
    +
    \frac{\norm{\operatorname{Proj}_{\mathcal F}r}^2}{q}
    <
    \frac23.
\end{align}
\end{lemma}

\begin{proof}
We first prove the result under the tight-frame assumption
\begin{align}\label{eq:tight-F1}
\tag{tight F1}
    \sum_i a_i^{} a_i^\top=P_{\cU},
\end{align}
which holds for the complete-graph EDG frame~\eqref{eq:EDGinverseframeconditions}.
In fact, in the tight case we shall prove the stronger inequality
\begin{align}\label{eq:meq1-tight-stronger}
    \tr_{\mathcal F}(\widetilde P)
    +
    \frac{\norm{r}^2}{q}
    <
    \frac23.
\end{align}

\paragraph{Notation.} Let
\begin{align}\label{defthetaetal}\tag{def-$\theta, \chi_2, \chi_3$}
    \theta_i:=w_i^2,
    \qquad
    \chi_2:=\sum_i \theta_i^2,
    \qquad
    \chi_3:=\sum_i \theta_i^3,
    \qquad
    q=1-\chi_2.
\end{align}
Tightness~\eqref{eq:tight-F1} along with $V^\top V = I_k, W^\top W = 1, V^\top W = 0$ gives
\begin{align}\label{eq:yetanotheronethree}
    \sum_i \theta_i=1,
    \qquad
    \sum_i v_i^{} v_i^\top=I_k,
    \qquad
    \sum_i w_i v_i=0.
\end{align}

\paragraph{Bounding $\tr_{\mathcal F}(\widetilde P)$.} Let \(U\in\R^{k\times2}\) have orthonormal columns spanning \(\mathcal F\). Let \({\Y} \in\R^{N\times2}\) have rows \(a_i^\top V U = v_i^\top U\). Then, by~\eqref{eq:yetanotheronethree},
\begin{align}\label{eq:aboutcalY}
    {\Y}^\top {\Y} = U^\top\bigg(\sum_i v_i^{} v_i^\top \bigg) U = I_2,
    \qquad
    {\Y}^\top w = \sum_i w_i y_i =0,
\end{align}
where \(w := (w_i)_i\in\R^N\) is a unit vector by~\eqref{eq:yetanotheronethree}. Also,
\[
    \tr_{\mathcal F}(\widetilde P)
    =
    \tr\!\left(\sum_i w_i^2 \cdot U^\top v_i^{} v_i^\top U\right)
    =
    \tr\!\left({\Y}^\top\Diag(\theta_i){\Y} \right).
\]

Let \(D_\theta:=\Diag(\theta_i) \in \Sym(N)\). 
Let \(w^\perp := \{x \in \mathbb{R}^N : x^\top w = 0\}\), and let $\mathcal W \in \mathbb{R}^{N \times N-1}$ have orthonormal columns spanning $w^\perp$.
Since \({\Y}\) has orthonormal columns contained in
\(w^\perp\) by~\eqref{eq:aboutcalY}, Ky Fan's variational principle~\citep[Cor.~4.3.39]{Horn1994} gives
\begin{align}\label{eq:yetanotheronefour}
    \tr_{\mathcal F}(\widetilde P)
    \le
    \mu_{N-2}+\mu_{N-1} = \tr({\mathcal W}^\top D_\theta {\mathcal W}) - (\mu_1 + \cdots + \mu_{N-3}),
\end{align}
where \(\mu_1\le\cdots\le\mu_{N-1}\) are the eigenvalues of ${\mathcal W}^\top D_\theta {\mathcal W}$, which is the compression of \(D_\theta\) to
\(w^\perp\). The trace of this compression is, using~\eqref{eq:yetanotheronethree},
\begin{align}\label{eq:yetanotheronefive}
    \tr({\mathcal W}^\top D_\theta {\mathcal W}) 
    = 
    \tr\big(D_\theta(I - w w^\top)\big) = \tr(D_\theta)-w^\top D_\theta w
    =
    1-\chi_2=q.
\end{align}

Let \(\theta_{(1)}\le\cdots\le\theta_{(N)}\) be the nondecreasing rearrangement of the
\(\theta_i\)'s.
By Cauchy interlacing~\citep[Cor.~4.3.17]{Horn1994}, the sum of the \(N-3\) smallest eigenvalues of ${\mathcal W}^\top D_\theta {\mathcal W}$ is at least $\sum_{i=1}^{N-3}\theta_{(i)}$.
Combining this with~\eqref{eq:yetanotheronefour} and~\eqref{eq:yetanotheronefive}, we get
\begin{align}\label{eq:meq1-trace-ptilde-bound}
    \tr_{\mathcal F}(\widetilde P)
    \le
    q-\sum_{i=1}^{N-3}\theta_{(i)}.
\end{align}

\paragraph{Bounding $\frac{\norm{r}^2}{q}$.}
Let \(\mathcal V \in \R^{N \times k}\) have rows $v_i^\top$. 
Tightness~\eqref{eq:yetanotheronethree} gives \({\mathcal V}^\top {\mathcal V}=I_k\)
and \({\mathcal V}^\top w=0\). Hence \({\mathcal V}^\top\) is a contraction and annihilates \(w\), so
\begin{equation}\label{eq:meq1-r-bound}
\begin{split}
    \norm{r}^2
    =
    \norm{{\mathcal V}^\top (w_i^3)_i}^2
    &=
    \norm{{\mathcal V}^\top (I - w w^\top) (w_i^3)_i}^2 \\
    &\le
    \norm{(I - w w^\top)(w_i^3)_i}^2
    =
   (w_i^3)_i^\top (I - w w^\top)(w_i^3)_i
    = 
    \chi_3-\chi_2^2.
\end{split}
\end{equation}

\paragraph{Wrapping up the tight-frame case.} Combining~\eqref{eq:meq1-trace-ptilde-bound} and~\eqref{eq:meq1-r-bound},
\[
    \tr_{\mathcal F}(\widetilde P)
    +
    \frac{\norm{r}^2}{q}
    \le
    q +\frac{\chi_3-\chi_2^2}{q} - \sum_{i=1}^{N-3}\theta_{(i)} .
\]
By Lemma~\ref{lem:simplex-ineq}, every $\theta$ in the simplex $\{\theta \in \R^N : \theta_i \geq 0, \sum_i \theta_i = 1\}$ with $q > 0$ satisfies
\begin{align}\label{eq:lemmassimplexinequ}
   q +\frac{\chi_3-\chi_2^2}{q} - \sum_{i=1}^{N-3}\theta_{(i)} \le\frac23.
\end{align}
Therefore,
\(
    \tr_{\mathcal F}(\widetilde P)
    +
    \frac{\norm{r}^2}{q}
    \le
    \frac23.
\)

To prove~\eqref{eq:meq1-tight-stronger}, it remains to rule out equality. If equality holds in inequality~\eqref{eq:lemmassimplexinequ}, then Lemma~\ref{lem:simplex-ineq} says
either \(\theta=e_j\), or \(\theta\) has exactly three nonzero entries, each equal to
\(1/3\). The first case cannot happen, since it contradicts~\eqref{eq:F2}:
\(
    \norm{a_j}^2\ge w_j^2=1.
\)
In the second case, write \(\Theta\) for the three indices with \(\theta_i=1/3\). Then by~\eqref{eq:yetanotheronethree},
\[
    r=\sum_{i\in \Theta}w_i^3v_i=\frac13\sum_{i\in \Theta}w_iv_i=0, \qquad \text{and} \qquad  \widetilde P=\frac13\sum_{i\in \Theta}v_i^{} v_i^\top.
\]
For every \(i\in \Theta\),~\eqref{eq:F2} and~\eqref{eq:viwinorm} give
\(
    \norm{v_i}^2
    =
    \norm{a_i}^2-w_i^2
    \le
    1-\eta-\frac13
    =
    \frac23-\eta.
\)
Consequently,
\[
    \tr_{\mathcal F}(\widetilde P)
    +
    \frac{\norm{r}^2}{q}
    =
    \tr_{\mathcal F}(\widetilde P)
    \le
    \tr(\widetilde P)
    =
    \frac13\sum_{i\in \Theta}\norm{v_i}^2
    \le
    \frac23-\eta
    <
    \frac23.
\]
Thus equality is impossible, and~\eqref{eq:meq1-tight-stronger} holds in the tight case.

\paragraph{Handling loose frames.} 
The extension from tight frames to the general upper-frame setting~\eqref{eq:F1}
is established in Lemma~\ref{lem:loose-frame-extension} of
Appendix~\ref{app:loose-frame}, which proves precisely
\eqref{eq:meq1-two-plane}. The argument reduces the upper-frame case to
the tight-frame case by adjoining atoms to complete the frame to a tight
one, and then observing that
\(
\tr_{\mathcal F}(\widetilde P)
    +
    \frac{\norm{\operatorname{Proj}_{\mathcal F}r}^2}{q}
\)
is monotone nondecreasing under the addition of atoms, with
\(V\), \(W\), and \(\mathcal F\) held fixed.
\end{proof}

We can now prove Theorem~\ref{thm:meq1-complete}.

\begin{proof}[Proof of Theorem~\ref{thm:meq1-complete}]
By Lemma~\ref{lem:rkdeficient}, it suffices to rule out full-rank non-global
second-order critical points. Suppose, for contradiction, that such a point exists.
Lemma~\ref{lem:endpoint-plane} gives a two-dimensional subspace \(\mathcal F\) satisfying
\[
    \tr_{\mathcal F}(\widetilde P)
    +
    \frac{\norm{\operatorname{Proj}_{\mathcal F}r}^2}{q}
    \ge
    \frac23.
\]
This contradicts Lemma~\ref{lem:m1-coordinate}. Hence no full-rank non-global
second-order critical point exists. Therefore every second-order critical point is
globally optimal.
\end{proof}

\section{Perspectives}

We conclude with a few open questions.

\begin{itemize}[leftmargin=2em]
    \item \textbf{Relaxing by one dimension?}
    Does the complete-graph s-stress have a benign landscape for all ground truths whenever \(k\ge \ell+1\)? 
    We hope that Theorem~\ref{thm:meq1-complete} and the techniques developed in
Sections~\ref{sec:directions} and~\ref{sec:schur-companion} provide some
guidance toward resolving this question.

    \item \textbf{Incomplete graphs?}
    To what extent can the results of this paper be extended to incomplete graphs, where only a subset of the pairwise distances is observed?
    One possible route is suggested by \citep[\S8.8]{Criscitiello2025thesis}. 
    A main obstacle, however, is that our complete-graph analysis relies crucially on the positive semidefiniteness of the stress matrix (Lemma~\ref{lem:stress-psd}). For incomplete graphs, the analogous stress matrix appears not to satisfy any comparable positivity property---even in a weak form, such as having only a small number of negative eigenvalues. 

We emphasize that the interesting regime is that of globally or universally
rigid frameworks, since only then does minimizing the s-stress correspond to a meaningful recovery
problem. Important structured examples include general-position
lateration frameworks~\citep{alfakih2013lateration}.
By contrast, benign landscape results for non-rigid graphs are often much easier to establish.\footnote{For example, it is straightforward to show that benignness is preserved when two graphs are glued together by identifying a single common vertex. 
    This immediately implies that all trees have benign landscapes when \(k=\ell\), and, combined with Theorem~\ref{thm:intro-informal}, yields benign landscapes for block graphs when $k \geq 2(\ell+1)$.}

    \item \textbf{Geometric meaning of the descent directions?}
    Is there a geometric interpretation of the descent directions underlying our analysis, at the level of motions of point clouds or their pairwise distances?  In many benign-landscape results, the relevant descent directions admit a simple geometric interpretation.  For example, in matrix completion and phase retrieval, the descent direction consists of moving the current configuration toward the target configuration, suitably aligned (see Section~\ref{sec:relatedwork}). Such directions, however, are insufficient for understanding the s-stress.

    A key ingredient of our analysis is that the stress matrix at a first-order critical point is positive semidefinite (Lemma~\ref{lem:stress-psd}), a property closely connected to universal rigidity. 
    Perhaps rigidity theory can provide a geometric explanation of the descent mechanism?
\end{itemize}

\section*{Acknowledgments}

\paragraph{AI:}
The author used GPT-5.5 Pro during the preparation of this work. The descent directions (Sections~\ref{sec:directions}), the overall proof strategy, and the proof of the \(m=1\) case (Section~\ref{sec:meq1}) were developed before any AI assistance. AI was subsequently used to help derive several key estimates, simplify and interpret parts of the arguments, and help with writing. The main AI-assisted mathematical contribution was Lemma~\ref{lem:best-response}, which is largely due to GPT-5.5 Pro.
The author assumes responsibility for all content.

\bibliographystyle{plainnat}
\bibliography{references}

\appendix
\section{Auxiliary inequalities for the codimension-one case}
\label{app:m1-ineq}

\subsection{A simplex inequality}
\label{app:simplex-ineq}

\begin{lemma}
\label{lem:simplex-ineq}
Let \(N\ge3\), and let \(\theta=(\theta_i)_{i=1}^N\) be a probability vector, i.e., $\theta_i \geq 0$ and $\sum_i \theta_i = 1$. Let
\(
    \theta_{(1)}\le\cdots\le\theta_{(N)}
\)
be its nondecreasing rearrangement, and set
\[
    \chi_2:=\sum_i\theta_i^2,
    \qquad
    \chi_3:=\sum_i\theta_i^3, \qquad q=1-\chi_2.
\]
Then
\begin{align}\label{eq:app-simplex-ineq}
     q^2 + \chi_3-\chi_2^2 - q \sum_{i=1}^{N-3}\theta_{(i)} \le\frac23 q.
\end{align}
Equality can occur only if, up to permutation,
\[
    \theta=e_1
    \qquad\text{or}\qquad
    \theta=(0,\ldots,0,1/3,1/3,1/3).
\]
\end{lemma}

\begin{proof}
Let \(x,y,z\) be the three largest entries of \(\theta\), and define $T := x+y+z = \sum_{i=N-2}^{N}\theta_{(i)}$.
Using $\sum_{i=1}^{N-3}\theta_{(i)} = 1 - T$ and rearranging,~\eqref{eq:app-simplex-ineq} is equivalent to
\begin{align}\label{eq:app-simplex-ineq-rear}
(1-\chi_2)\left(\frac23-T\right)+\chi_2-\chi_3\ge0.
\end{align}
To show~\eqref{eq:app-simplex-ineq-rear}, we distinguish two cases based on the size of $T$.

\paragraph{Case 1.} If \(T\le2/3\), then the first term in~\eqref{eq:app-simplex-ineq-rear} is nonnegative,
and
\[
    \chi_2-\chi_3=\sum_i\theta_i^2(1-\theta_i)\ge0.
\]
Thus~\eqref{eq:app-simplex-ineq-rear} holds. Equality in this case would force each
\(\theta_i\in\{0,1\}\), hence \(\theta=e_j\); but then \(T=1\), contradicting
\(T\le2/3\). Hence equality cannot occur in this case.

\paragraph{Case 2.} Now assume \(T>2/3\). Write the left-hand side of
\eqref{eq:app-simplex-ineq-rear} as
\begin{equation}\label{eq:cimepl1}
\begin{split}
    (1-\chi_2)\left(\frac23-T\right)+\chi_2-\chi_3 
    &=
    \frac23-T+\sum_i\theta_i^2\left(T+\frac13-\theta_i\right) \\
    &\geq
    \frac23 - T + \sum_{t\in\{x,y,z\}}t^2\left(T+\frac13-t\right),
\end{split}
\end{equation}
since $T+\frac13-\theta_i
    > 1-\theta_i
    \ge0$.
Let
\(
    \PP:=xy+xz+yz, \text{and }
    \RR:=xyz.
\)
A direct calculation gives
\begin{align}\label{eq:cimepl2}
    3\left[
    \frac23 - T +
    \sum_{t\in\{x,y,z\}}t^2\left(T+\frac13-t\right)
    \right]
    =
    (1-T)(2-T)+\PP(3T-2)-9\RR.
\end{align}
Using the elementary inequalities
\(
    9\RR\le T\PP,
\)
\(\PP\le T^2/3\) and $T \leq 1$,
\begin{equation}\label{eq:cimepl3}
\begin{split}
    (1-T)(2-T)+\PP(3T-2)-9\RR
    &\ge
    (1-T)(2-T)+\PP(2T-2) \\
    &=
    (1-T)(2-T-2\PP) \\
    &\geq
    (1-T)\left(2-T-\frac{2T^2}{3}\right) \\
    &\geq
    \frac13 (1-T) \geq 0.
\end{split}
\end{equation}
The final inequality comes from $2-T-\frac{2T^2}{3} \geq \frac13$ for \(T\in[2/3,1]\). Combining~\eqref{eq:cimepl1},~\eqref{eq:cimepl2} and~\eqref{eq:cimepl3} gives
\eqref{eq:app-simplex-ineq-rear}.

Finally, equality in the second case can occur only if \(T=1\) and
\(9\RR=T\PP\). This forces either two of \(x,y,z\) to be zero, giving \(\theta=e_j\), or
\(x=y=z=1/3\), giving
\[
    \theta=(0,\ldots,0,1/3,1/3,1/3),
\]
up to permutation.
\end{proof}

\subsection{Passing from tight frames to upper frames}
\label{app:loose-frame}

The proof of Lemma~\ref{lem:m1-coordinate} establishes the inequality
\[
    \tr_{\mathcal F}(\widetilde P)+\frac{\norm{r}^2}{q}<\frac23
\]
under the tight-frame assumption
\(
    \sum_i a_i^{} a_i^\top=P_{\cU}.
\)
For the original upper-frame hypothesis~\eqref{eq:F1}, we have the following.

\begin{lemma}
\label{lem:loose-frame-extension}
In the same setting as Lemma~\ref{lem:m1-coordinate},
\[
    \tr_{\mathcal F}(\widetilde P)
    +
    \frac{\norm{\operatorname{Proj}_{\mathcal F}r}^2}{q}
    <
    \frac23.
\]
\end{lemma}

\begin{proof}
Assume \(\eta<1\).\footnote{If \(\eta=1\), then~\eqref{eq:F2} forces \(a_i=0\) for every \(i\).
Consequently,
\(
    \widetilde P=0,
    r=0,
    q=1,
\)
and the desired inequality is immediate.}
Let
\[
    \mathcal R
    :=
    P_{\cU}-\sum_i a_i a_i^\top
    \succeq0.
\]
Choose a rank-one decomposition
\[
    \mathcal R=\sum_j \tilde a_j\tilde a_j^\top
\]
such that every added atom satisfies
\(
    \norm{\tilde a_j}^2\le 1-\eta.
\)
This can always be achieved by first taking any finite rank-one decomposition of
\(\mathcal R\), and then splitting each rank-one term into sufficiently many equal
pieces. After adjoining the atoms \(\tilde a_j\), the frame becomes tight:
\[
    \sum_i a_i a_i^\top+\sum_j \tilde a_j\tilde a_j^\top=P_{\cU},
\]
while the atom-norm condition~\eqref{eq:F2} remains valid.

Fix \(V,W\), and a two-dimensional subspace \(\mathcal F\subseteq\R^k\).
For any intermediate frame obtained while adjoining the atoms one at a time,
define \(\widetilde P,r,q\) by~\eqref{eq:meq1-def-ptilde-r-q}, and set
\[
    \mathcal Q_{\mathcal F}
    :=
    \tr_{\mathcal F}(\widetilde P)
    +
    \frac{\norm{\operatorname{Proj}_{\mathcal F}r}^2}{q}.
\]

\paragraph{Denominators are positive.} We first note that every denominator \(q\) arising in this process is strictly
positive. Let \(q_{\rm final}\) denote the value after all added atoms have been
adjoined. Since the completed frame is tight,
\[
    \sum_i w_i^2=1,
\]
where the index ranges over both the original and added atoms. Moreover,
\eqref{eq:F2} gives
\[
    \max_i w_i^2\le \max_i \norm{a_i}^2\le 1-\eta.
\]
Hence
\(
    \sum_i w_i^4
    \le
    \left(\max_i w_i^2\right)\sum_i w_i^2
    \le
    1-\eta,
\)
and therefore
\[
    q_{\rm final}
    =
    1-\sum_i w_i^4
    \ge
    \eta
    >
    0.
\]
Since adjoining an atom only decreases \(q\), every intermediate value of \(q\)
is at least \(q_{\rm final}\), and is therefore strictly positive.

\paragraph{\(\mathcal Q_{\mathcal F}\) is monotone.} We next show that \(\mathcal Q_{\mathcal F}\) is monotone nondecreasing as atoms
are adjoined. 
Consider one additional atom \(\tilde a\in\cU\), and write
\[
    \tilde v:=V^\top\tilde a\in\R^k,
    \qquad
    \tilde w:=W^\top\tilde a\in\R.
\]
Adding the atom $\tilde a$ updates $\widetilde P, \operatorname{Proj}_{\mathcal F}r$ and $q$ as follows:
\[
    \widetilde P\mapsto \widetilde P+{\tilde w}^2 {\tilde v} {\tilde v}^\top,
    \qquad
    \operatorname{Proj}_{\mathcal F} r \mapsto \operatorname{Proj}_{\mathcal F}r + {\tilde w}^3 \operatorname{Proj}_{\mathcal F} \tilde v,
    \qquad
    q\mapsto q-{\tilde w}^4.
\]
Thus, letting \(\mathcal Q_{\mathcal F}^{\rm new}\) denote
\(\mathcal Q_{\mathcal F}\) after adjoining the atom \(\tilde a\),
\[
\begin{aligned}
    \mathcal Q_{\mathcal F}^{\rm new}-\mathcal Q_{\mathcal F}
    &=
    \tr_{\mathcal F}({\tilde w}^2 {\tilde v} {\tilde v}^\top)
    +
    \frac{\norm{\operatorname{Proj}_{\mathcal F}r + {\tilde w}^3 \operatorname{Proj}_{\mathcal F} \tilde v}^2}{q-{\tilde w}^4}
    -
    \frac{\norm{\operatorname{Proj}_{\mathcal F}r}^2}{q} \\
     &=
     {\tilde w}^2 \norm{\operatorname{Proj}_{\mathcal F} \tilde v}^2
    +
    \frac{\norm{\operatorname{Proj}_{\mathcal F}r + {\tilde w}^3 \operatorname{Proj}_{\mathcal F} \tilde v}^2}{q-{\tilde w}^4}
    -
    \frac{\norm{\operatorname{Proj}_{\mathcal F}r}^2}{q} \\
    &=
    \frac{\norm{q{\tilde w} \cdot {\operatorname{Proj}_{\mathcal F} \tilde v}+{\tilde w}^2 \cdot {\operatorname{Proj}_{\mathcal F} r}}^2}{q(q-{\tilde w}^4)}
    \ge0,
\end{aligned}
\]
where both denominators are positive by the preceding paragraph.

Therefore \(\mathcal Q_{\mathcal F}\) for the original upper frame is bounded
above by its value for the completed tight frame. For the latter, the stronger
tight-frame estimate from the proof of Lemma~\ref{lem:m1-coordinate} gives
\[
    \mathcal Q_{\mathcal F}
    \le
    \tr_{\mathcal F}(\widetilde P)
    +
    \frac{\norm{r}^2}{q}
    <
    \frac23.
\]
Hence the same strict bound holds for the original upper frame.
\end{proof}

\end{document}